\documentclass{article}

\usepackage{amsmath,amsthm,amsfonts}
\usepackage{graphics,color,epsfig}
\newtheorem{theorem}{Theorem}[section]
\newtheorem{lemma}[theorem]{Lemma}

\newtheorem{conj}[theorem]{Conjecture}
\newtheorem{question}[theorem]{Question}
\theoremstyle{definition}

\newtheorem{example}[theorem]{Example}

\newtheorem*{ack}{Acknowledgement}

\newcommand\noproof{\hfill$\Box$}

\newcommand\ssequiv{\cite[Subsection 2.4]{BRsparse}}

\newcommand\webcite[1]{\hfil\penalty0\texttt{\def~{\~{}}#1}\hfill\hfill}
\newcommand\arxiv[1]{\webcite{arXiv:#1.}}

\newcommand\floor[1]{\lfloor #1\rfloor}
\newcommand\norm[1]{||#1||}
\newcommand\bb[1]{\bigl(#1\bigr)}

\newcommand\E{{\mathbb E{}}}
\newcommand\de{d_{\mathrm{sub}}}
\newcommand\detl{{\tilde d}_{\mathrm{sub}}}
\newcommand\dcn{d_{\mathrm{cn}}}
\newcommand\dloc{d_{\mathrm{loc}}}
\newcommand\sss{{\mathcal S}}
\newcommand\eps{\varepsilon}
\newcommand\isom{\cong}

\newcommand\emb{\mathrm{emb}}

\newcommand\ka{\kappa}
\newcommand\la{\lambda}
\newcommand\La{\Lambda}
\newcommand\Ga{\Gamma}
\newcommand\A{\mathcal{A}}
\newcommand\C{\mathcal{C}}
\newcommand\D{\mathcal{D}}
\newcommand\F{\mathcal{F}}
\newcommand\M{\mathcal{M}}
\newcommand\Mk{B_k}

\newcommand\T{\mathcal{T}}
\renewcommand\P{\mathcal{P}}
\newcommand\Tr{\mathcal{T}^{\mathrm r}}
\newcommand\Trr{\mathcal{T}^{\mathrm{rr}}}
\newcommand\Gr{\mathcal{G}^{\mathrm r}}

\newcommand\dc{d_\mathrm{cut}}
\newcommand\dedit{d_\mathrm{edit}}
\newcommand\dH{d_\mathrm{H}}
\newcommand\dM{d_\mathrm{match}}
\newcommand\dPM{d_\mathrm{p-m}}
\newcommand\dcnM{d_\mathrm{cn-m}}
\newcommand\dTV{d_\mathrm{TV}}
\newcommand\dcG{{\widehat d}_\mathrm{cut}}
\newcommand\dP{d_\mathrm{part}}
\newcommand\cn[1]{||#1||_{\mathrm{cut}}}
\newcommand\cc{{\mathrm c}}% for complement
\newcommand\rr{{\mathrm r}}
\newcommand\dd{\,d}

\newcommand\RR{{\mathbb R}}
\newcommand\R{{\mathbb R}}

\newcommand\iid{i.i.d.}
\renewcommand\Pr{\mathbb{P}}

\newcommand\pto{\overset{\mathrm{p}}{\to}}

\newcommand\tpi{{\widetilde\pi}}
\newcommand\bp{{\mathfrak X}}
\newcommand\kar{{\ka_\mathrm{r}}}
\newcommand\kac{{\ka_\mathrm{c}}}
\newcommand\Po{{\mathrm{Po}}}
\newcommand\mut{{\widetilde\mu}}
\newcommand\vv{{\bf v}}
\newcommand\lamax{{\la_{\mathrm{max}}}}
\newcommand\ts{{\tilde s}}
\newcommand\ttt{{\tilde t}}
\newcommand\diff{{\Delta}}
\newcommand\op{o_{\mathrm{p}}}
\newcommand\Op{O_{\mathrm{p}}}

\begin{document}
\title{Sparse graphs: metrics and random models}
\date{December 14, 2008; revised 07 February 2010}

\author{B\'ela Bollob\'as%
\thanks{Department of Pure Mathematics and Mathematical Statistics,
Wilberforce Road, Cambridge CB3 0WB, UK and
Department of Mathematical Sciences, University of Memphis, Memphis TN 38152, USA.
E-mail: {\tt b.bollobas@dpmms.cam.ac.uk}.}
\thanks{Research supported in part by NSF grants DMS-0906634,
 CNS-0721983 and CCF-0728928,
and ARO grant W911NF-06-1-0076}
\and Oliver Riordan%
\thanks{Mathematical Institute, University of Oxford, 24--29 St Giles', Oxford OX1 3LB, UK.
E-mail: {\tt riordan@maths.ox.ac.uk}.}
}
\maketitle

\begin{abstract}
Recently, Bollob\'as, Janson and Riordan introduced
a family of random graph models producing inhomogeneous
graphs with $n$ vertices
and $\Theta(n)$ edges whose distribution is characterized
by a kernel, i.e., a symmetric measurable function
$\ka:[0,1]^2 \to [0,\infty)$. To understand these models,
we should like to know when different kernels $\ka$ give
rise to `similar' graphs, and, given a real-world network,
how `similar' is it to a typical graph $G(n,\ka)$
derived from a given kernel $\ka$.

The analogous questions for dense graphs, with $\Theta(n^2)$ edges,
are answered by recent results of Borgs, Chayes, Lov\'asz, S\'os, Szegedy
and Vesztergombi, who showed that several natural metrics on graphs are
equivalent, and moreover that any sequence of graphs converges in each
metric to a graphon, i.e., a kernel taking values in $[0,1]$.

Possible generalizations of these results to graphs with $o(n^2)$
but $\omega(n)$ edges are discussed in a companion paper~\cite{BRsparse};
here we focus only on graphs with $\Theta(n)$ edges,
which turn out to be much harder to handle. Many new phenomena occur,
and there are a host of plausible metrics to consider; many of these
metrics suggest new random graph models, and vice versa.

\end{abstract}

\section{Introduction}\label{sec_es}
In a series of papers, Borgs, Chayes, Lov\'asz, S\'os, Szegedy and
Vesztergombi (see
\cite{BCLSV:homcount,BCLSSV:stoc,LSz1,LSgenQT,BCLSV:1,BCLSV:2} and the
references therein) introduced several natural metrics for
graphs, and showed that they are equivalent, in that if $(G_n)$
is a sequence of graphs with $|G_n|\to\infty$, then if $(G_n)$
is Cauchy with respect to one of these metrics then it is
Cauchy with respect to all of them. Moreover, there is a natural
completion of the space of graphs with respect to any of
these metrics, consisting of (equivalence classes of)
{\em graphons}, i.e., symmetric measurable functions
$\ka:[0,1]^2\to [0,1]$.
Throughout this paper we assume without loss of generality that $G_n$
has $n$ vertices; we do not require $G_n$ to be defined for all $n$,
but only for a sequence $n_i\to\infty$.  While the results just mentioned
apply to all sequences $(G_n)$,
they are meaningful only for {\em dense} graphs, where $e(G_n)=\Theta(n^2)$.
More precisely, any sequence with $e(G_n)=o(n^2)$ converges
to the zero graphon.

A different connection between graphs and objects related
to graphons arises in the work of Bollob\'as, Janson and Riordan~\cite{BJR}.
Throughout this paper, by a {\em kernel} $\ka$ we shall mean a symmetric integrable
function $\ka:[0,1]^2\to [0,\infty)$; note that
graphons are a special case of kernels. Roughly speaking, in~\cite{BJR}
an arbitrary kernel $\ka$ was used to define a sparse inhomogeneous
random graph $G(n,\ka)=G_{1/n}(n,\ka)$, although the details
are rather involved.

In~\cite{BRsparse} we extended the definitions of three of the
metrics mentioned above, the cut metric $\dc$, the count (or subgraph) metric $\de$,
and the partition metric $\dP$, to sparse graphs. In each
case one fixes a normalizing function $p=p(n)$ and adapts
the definition of the metric to graphs with $e(G_n)=\Theta(p n^2)$;
for the details see the definitions in the relevant sections here.
In addition to discussing the relationships between the different metrics,
we also discussed the close connection between metrics
and random graph models, concentrating on the case
where $p(n)$ is chosen so that $np\to\infty$. Here we
shall continue this investigation, but now considering
the case $p=1/n$.

When studying, for example, the random graph $G(n,p)$, there are
many possibilities for $p$ as a function of $n$; which is
most natural depends on what kind of properties one is interested
in. Nevertheless, there are two canonical ranges of particular interest:
the {\em dense} case, $p=\Theta(1)$, and the {\em (extremely) sparse} case,
$p=\Theta(1/n)$, the minimum sensible density.
Here we are not only
studying random graphs, but it is still true that the most natural
special cases are the densest graphs, those with $\Theta(n^2)$
edges, studied by Lov\'asz and Szegedy~\cite{LSz1} and Borgs,
Chayes, Lov\'asz, S\'os and Vesztergombi~\cite{BCLSV:1,BCLSV:2}, for
example, and the sparsest graphs, those with $\Theta(n)$ edges, as
studied by Bollob\'as, Janson and Riordan~\cite{BJR}.
Here we consider the second range, taking $p=p(n)=1/n$
as our normalizing density.

One might expect that graphs with $\Theta(n)$ edges are somehow simpler
than denser graphs, but in fact the reverse is often the case,
particularly for the random graph $G(n,p)$. As a trivial example,
note that there is significant variation in the vertex degrees in $G(n,c/n)$,
while the degrees in $G(n,p)$ are concentrated around their mean
if $np\to\infty$. For this reason, we expect graphs with $\Theta(n)$
edges to be much harder to work with in the present context, which
turns out to be the case. Indeed, as we shall see, hardly
any of the results in~\cite{BRsparse} apply to such graphs.

One advantage of the extremely sparse case is that there is a unique
natural normalization: except where explicitly indicated
otherwise, in this paper we fix $p=1/n$ as our normalizing function.
We shall discuss several metrics in turn, starting with the cut metric.
Before doing so, let us recall a few definitions from (for example)~\cite{BRsparse}.

Throughout this paper, by a {\em kernel} we mean an integrable function $\ka:[0,1]^2\to[0,\infty)$
with $\ka(x,y)=\ka(y,x)$ for all $x$, $y$.
A {\em rearrangement} of a kernel $\ka$ is any kernel $\ka^{(\tau)}$ defined by
$\ka^{(\tau)}(x,y)=\ka(\tau(x),\tau(y))$, where $\tau:[0,1]\to[0,1]$ is a measure-preserving
bijection. We write $\ka\approx \ka'$ if there is a rearrangement $\ka^{(\tau)}$
of $\ka$ with $\ka'=\ka^{(\tau)}$ a.e.

A kernel $\ka$ is of {\em finite type} if there is a finite partition $(A_1,\ldots,A_n)$
of $[0,1]$ such that $\ka$ is constant on each set $A_i\times A_j$.
Given a graph $G_n$ with $n$ vertices and a normalizing function $p=p(n)$,
we write $\ka_{G_n}$ for the finite-type kernel associated to $G_n$, defined by
partitioning $[0,1]$ into $n$ intervals $I_i$ of length $1/n$ and setting
$\ka_{G_n}$ equal to $1/p$ on $I_i\times I_j$ if $ij\in E(G_n)$, and to equal
to $0$ otherwise. Note that the definition of $\ka_{G_n}$ depends
on our normalizing function $p=p(n)$.

Given subsets $U$, $W$ of $V(G_n)$, we write $e(U,W)=e_{G_n}(U,W)$ for the number
of edges of $G$ from $U$ to $W$, i.e., the number of ordered pairs
$(u,w)$ with $u\in U$, $w\in W$ and $uw\in E(G_n)$. Suppressing the dependence
on $G_n$, we write
\begin{equation}\label{dpdef}
 d_p(U,W) = \frac{e_{G_n}(U,W)}{p|U||W|}
\end{equation}
for the normalized density of edges from $U$ to $W$ in $G_n$.

As in~\cite{BRsparse},
given a kernel $\ka$ and a normalizing function $p=p(n)$,
we write
$G_p(n,\ka)$ for the random graph defined by choosing vertex types
$x_1,\ldots,x_n$ independently and uniformly from $[0,1]$,
and, given these types, joining each pair $\{i,j\}$ of vertices
with probability $\min\{p\ka(x_i,x_j),1\}$, independently
of all other pairs. When $p=1/n$ this is a special
case of the sparse inhomogeneous model of Bollob\'as,
Janson and Riordan~\cite{BJR}; in~\cite{BJR} the sequence $x_1,\ldots,x_n$
is not assumed to be \iid, so the model there is much more general.
On the other hand, in~\cite{BJR} there are certain technical
assumptions, including that $\ka$ is continuous almost everywhere. These
assumptions are not needed here, since the \iid\ sequence
case is always well behaved; see the discussion in~\cite{BJRclust} or~\cite{BJRcutsub}.
When $p=1$ and $\ka$ is bounded by $1$, then $G_p(n,\ka)$
is what is called a {\em $\ka$-random graph}
by Lov\'asz and Szegedy~\cite{LSz1}.

Often in what follows we consider sequences $(G_n)$
of random graphs, i.e., sequences of probability
distributions on $n$-vertex graphs. In general, there
is no canonical coupling between these distributions
for different $n$, so formally we should only consider
convergence in probability. However, in many cases
the error bounds one obtains are strong enough to give
almost sure convergence for any coupling, and one
can in any case ensure almost sure convergence by passing
to a suitable subsequence. Since the relevant
`in probability' notions of (for example) Cauchy sequences
are perhaps unfamiliar and distracting, we shall often
implicitly fix a coupling and consider almost sure convergence instead.

As usual, when discussing random graphs we say that a sequence
of events $E_n$ holds {\em with high probability}, or {\em whp},
if $\Pr(E_n)\to 1$ as $n\to\infty$. We write $X_n\pto c$ to denote
convergence in probability. If $(X_n)$ is a sequence of random
variables and $f(n)$ a function, then $X_n=\op(f(n))$ means $X_n/f(n)\pto 0$.

\section{The cut metric and Szemer\'edi's Lemma}

Let us briefly recall the definitions of the cut norm of Frieze and
Kannan~\cite{FKquick}, and the cut metric, defined for kernels and
dense graphs by Borgs, Chayes, Lov\'asz, S\'os and
Vesztergombi~\cite{BCLSV:1}, and adapted to sparse graphs
in~\cite{BRsparse}.

Given an integrable function $\ka:[0,1]^2\to \R$, its {\em cut norm} is
\begin{equation}\label{cndefST}
 \cn{\ka} = \sup_{S,T\subset [0,1]} \left|\int_{S\times T} \ka(x,y) \dd x\dd y\right|,
\end{equation}
where the supremum is over all pairs of measurable subsets of $[0,1]$.
The {\em cut metric} is defined for kernels by
\begin{equation*} %\label{dcdef1}
 \dc(\ka_1,\ka_2) = \inf_{\ka_2'\approx \ka_2} \cn{\ka_1-\ka_2'},
\end{equation*}
where the infimum is over all rearrangements of $\ka_2$.
The cut metric is extended to graphs by mapping a graph $G_n$ to the corresponding
finite-type kernel $\ka_{G_n}$. Note that this mapping
depends on the normalizing function $p=p(n)$, so when applying the cut metric
to graphs we should more properly speak of the {\em $p$-cut metric}.
However, all our metrics will depend on the normalizing function $p$,
so most of the time we shall not indicate this dependence.

In the dense and intermediate ranges, one of the key results
used in the study of the cut metric is some form of Szemer\'edi's
Lemma~\cite{Szem}. In the extremely sparse setting, there is no
way to apply Szemer\'edi's Lemma: the `bounded density' assumption
considered in~\cite[Section 4]{BRsparse} can only be satisfied if $e(G_n)=o(n)$,
and there is no reasonable way to define an $(\eps,p)$-regular partition
so that such a thing exists at all!
Correspondingly, many of the nice properties of the cut metric
fail when $p=1/n$, as we shall now see.

Given a graph $G$ and a kernel $\ka$, let $\dc(G,\ka)=\dc(\ka_G,\ka)$.
As shown in~\cite{BRsparse}, when $np\to\infty$
then, under suitable mild assumptions, any sequence $(G_n)$
had a subsequence converging to some (bounded) kernel. Here
such convergence is impossible, except in the trivial
case where $\ka=0$ almost everywhere (in which case
$\dc(G_n,\ka)\to 0$ simply says that $e(G_n)=o(n)$).
This is easy to see for bounded kernels (using~\cite[Lemma~4.2]{BRsparse}),
but in fact holds for arbitrary kernels.

%\marginal{UPDATE REFS TO THIS RESULT!!!}
\begin{theorem}
Set $p=1/n$, let $\ka$ be a symmetric measurable function on $[0,1]^2$ with
$0<\int\ka<\infty$, and let $(G_n)$ be a sequence of graphs with $|G_n|=n$.
Then $\dc(G_n,\ka)$ is bounded away from zero.
\end{theorem}
\begin{proof}
Suppose not; then, passing to a subsequence, we have $\dc(G_n,\ka)\to 0$.
Hence there are rearrangements $\ka^{(\tau_n)}$ of $\ka$ such that 
\begin{equation}\label{cnt}
 \cn{\ka_{G_n}-\ka^{(\tau_n)}}\to 0.
\end{equation}
Note that
\[
\dc(G_n,\ka)\ge \left|\int\ka_{G_n}-\int\ka\right|= \left|\frac{2e(G_n)}{n}-\int\ka\right|,
\]
so $e(G_n)/n\to \int\ka/2$. In particular, $G_n$ has $\Theta(n)$ edges.

Let $M_n$ be a largest matching in $G_n$.
We claim that there is a constant $c>0$ such that, for $n$ large enough,
$M_n$ contains at least $cn$ edges. Otherwise, passing to a subsequence,
we may assume that $|M_n|/n\to 0$. Writing $A_n$ for the vertex set of $M_n$,
and $B_n$ for its complement, we have $e(B_n,B_n)=0$. 
Let $X_n$ be the subset
of $[0,1]$ corresponding to $B_n$ under the rearrangement $\tau_n$.
Then, from \eqref{cnt}, $\int_{X_n\times X_n}\ka\to 0$.
Writing $\mu$ for Lebesgue measure, we have $\mu(X_n)=|B_n|/n\to 1$,
so from basic properties
of integration it follows that $\int_{X_n\times X_n}\ka \to \int_{[0,1]^2}\ka$,
which is positive by assumption. This contradiction proves the claim.

Fix $c>0$ for which the claim above holds. Since $\ka$ is integrable,
we have $\int \ka 1_{\{\ka>C\}}\to 0$ as $C\to\infty$,
where $1_{\{\ka>C\}}: [0,1]^2\to \{0,1\}$
is the indicator function of the event that $\ka(x,y)>C$. In particular,
there is a $C<\infty$ with $\int \ka 1_{\{\ka>C\}}\le c/4$.
Fix an $n$ with $n>4C/c$, noting
that if $S\subset [0,1]^2$ satisfies $\mu(S)\le 1/n$, then
\begin{equation}\label{smsm}
 \int_S \ka \le C\mu(S) + \int \ka 1_{\{\ka>C\}} \le C/n+c/4 \le c/2.
\end{equation}
Choosing $n$ large enough, we may assume from \eqref{cnt} that
there is a $\ka'\approx\ka$ with
\begin{equation}\label{cnc}
 \cn{\ka_{G_n}-\ka'}\le c/25.
\end{equation}

Let $\{u_1w_1,\ldots,u_rw_r\}$ be a matching in $G_n$ with $r\ge cn$; such
a matching exists by our claim. Let $U=\{u_i\}$ and $W=\{w_i\}$.
Identifying subsets of $V(G)$ with subsets of $[0,1]$ in the natural way,
from \eqref{cnc} we have
\[
 \left| \int_{U\times W}\ka'  - \frac{e_{G_n}(U,W)}{n} \right| \le c/25.
\]
Let $U'$ be a random subset of $U$ obtained by selecting each vertex
independently with probability $1/2$, and let $W'$ be the complementary
subset of $W$, defined by $W'=\{w_i: u_i\notin U_i\}$.
The edges of our matching never appear as edges from $U'$ to $W'$.
On the other hand, any other edge $u_iw_j$, $i\ne j$,
from $U$ to $W$ has probability $1/4$ of appearing. Hence,
\[
 \E\bb{e_{G_n}(U',W')} = e_{G_n}(U,W)/4 - r/4.
\]
Similarly, writing $S\subset [0,1]^2$ for the union of the $r$
$1/n$-by-$1/n$ squares corresponding to the edges $u_iw_i$,
we have
\[
 \E\left( \int_{U'\times W'} \ka' \right) = \frac{1}{4}\int_{U\times W}\ka'
  - \frac{1}{4}\int_S \ka'.
\]
Combining the last three displayed equations using the triangle
inequality, and noting that $\mu(S)=r/n^2\le 1/n$, it follows that
\begin{eqnarray*}
 \left| \E\left( \int_{U'\times W'} \ka' \right)
     -\frac{1}{n}\E\bb{e_{G_n}(U',W')}\right|
 &\ge& \frac{r}{4n}-\frac{1}{4}\int_S \ka' - c/100 \\
 &\ge& c/4-c/8-c/100 > c/16,
\end{eqnarray*}
using \eqref{smsm}.
On the other hand, from \eqref{cnc},
\[
 \left |\int_{U'\times W'} \ka' - \frac{e_{G_n}(U',W')}{n} \right| \le c/25
\]
always holds, which implies a corresponding upper bound on the difference
of the expectations. Since $c/25<c/16$, we obtain a contradiction, completing the
proof.
\end{proof}

The argument above in fact shows much more.
\begin{theorem}\label{noCauchy}
With $p=1/n$, a sequence $(G_n)$ of graphs with $|G_n|=n$
is Cauchy with respect to $\dc$ if and only if $e(G_n)=o(n)$.
\end{theorem}
\begin{proof}
If $e(G_n)=o(n)$ then $(G_n)$ is trivially Cauchy, so we may
assume that $(G_n)$ is Cauchy.

On the one hand, if there is some $c>0$ such that infinitely many
$G_n$ contain a matching of size at least $cn$ then,
passing to a subsequence, we may assume that all $G_n$ do.
The argument above then
shows that for any kernel $\ka$, 
if $n$ is large enough then $\dc(\ka_{G_n},\ka)>c/25$. Applying
this with $\ka=\ka_{G_m}$ shows that $(G_n)$ cannot be
Cauchy.

On the other hand, if the largest matching in $G_n$
has size $o(n)$, then $G_n$ contains $n-o(n)$ vertices
spanning no edges, which implies that $\liminf \dc(\ka,G_n)\ge \int \ka$
for any fixed kernel $\ka$. Taking $\ka=\ka_{G_m}$, since $(\ka_{G_n})$
is Cauchy it follows that $\int\ka_{G_m}\to 0$ as $m\to\infty$,
i.e., that $e(G_n)=o(n)$.
\end{proof}

Theorem~\ref{noCauchy} shows that one cannot hope
to extend the results of Borgs, Chayes, Lov\'asz, S\'os and
Vesztergombi~\cite{BCLSV:1,BCLSV:2} for the dense
version of $\dc$, or those
of~\cite{BRsparse} for the sparse version with $np\to\infty$,
to the present extremely sparse case. It may still make sense, however, to use
$\dc$ as a measure of the similarity of two graphs $G_1$, $G_2$ with the same
number $n$ of vertices. For this purpose, as
in the denser context, there is a more natural
metric $\dcG(G_1,G_2)$, defined as the minimum over graphs $G_2'$
with $V(G_2')=V(G_1)$ and $G_2'$ isomorphic to $G_1$ of
the maximum over $S,T\subset V(G_1)$ of $|e_{G_1}(S,T)-e_{G_2'}(S,T)|/n$.
(This corresponds to only allowing rearrangements of $[0,1]$
that map vertices, i.e., intervals $((i-1)/n,i/n]$, into vertices.)
In the extremely sparse case, it is not so clear
what it means for $\dcG(G_n,G_n')$ to tend to zero, if $G_n$, $G_n'$
are graphs with $n$ vertices. For example, consider
the following concrete question:

Fix $c>0$, and let $G_n$ and $G_n'$ be independent instances of the
Erd\H os--R\'enyi random graph $G(n,c/n)$. Does
the expected value of $\dcG(G_n,G_n')$ tend to $0$ as $n\to\infty$?

If we consider graphs that are even
a little denser, i.e., $G(n,p)$ with $np\to\infty$, then the 
answer is trivially yes,
defining $\dcG$ with respect to this density, of course. Indeed, there
is no need to rearrange the vertices of $G_n'$ in this case: it is immediate from
Chernoff's inequality that whp every cut in $G(n,p)$ has size within $o(pn^2)$
of its expectation. As we shall see, in the extremely sparse case
the situation is rather different.

Firstly, with $p=c/n$, rearrangement is certainly necessary:
we must match all but $o(n)$ of the $(e^{-c}+\op(1))n$ isolated vertices
of $G_n$ with isolated vertices of $G_n'$.
In fact, matching up almost all the small components of $G_n$ with
isomorphic small components of $G_n'$, it is not hard to see
that the question above reduces to a question about the giant
component of $G(n,c/n)$. In particular, if $c\le 1$, then the answer
is yes, for the rather uninteresting reason that $G_n$
can be made isomorphic to $G_n'$ by adding and deleting $\op(n)$ edges.
As we shall see, this is the only possibility for a positive answer!

For two graphs $G_1$, $G_2$ with the same number of vertices,
the {\em (normalized) edit distance} between $G_1$ and $G_2$ is the minimum
number of edge changes (additions or deletions) needed
to turn one of the graphs into a graph isomorphic to the other,
divided by $pn^2$:
\begin{equation}\label{editdef}
 \dedit(G_1,G_2) = \frac{1}{pn^2}\min\{ |E(G_1)\diff E(G_2')| : G_2'\isom G_2\}.
\end{equation}
Usually, one would leave the edit distance unnormalized; here we normalize
for consistency with our notation for $\dcG$.
It seems that Axenovich, K\'ezdy and
Martin~\cite{AKM} were the first to define the edit distance
explicitly, although implicitly the notion had been used much
earlier, e.g., by Erd\H{o}s~\cite{Erd68} and Simonovits~\cite{Sim68}
in 1966, and in many subsequent papers.
If $|G_1|=|G_2|=n$, then, trivially, $\dcG(G_1,G_2)\le 2\dedit(G_1,G_2)$.
In general, $\dcG$ may be {\em much} smaller than $\dedit$;
for example, in the dense case where $p=1$, two independent
instances of the random graph $G(n,1/2)$ are very close in $\dcG$, but
far apart in $\dedit$. One can construct sparse examples by `cheating': take
two instances of $G(\sqrt{n},1/2)$ padded with $n-\sqrt{n}$ isolated vertices.

Although $\dcG$ and $\dedit$ are in general very different, in the
extremely sparse case it turns out that they are closely
related, at least for well behaved graphs.
\begin{lemma}\label{dcde}
Let $p=1/n$, and let $(G_n)$ and $(G_n')$ be sequences of graphs with $O(n)$ edges,
with $|G_n|=|G_n'|=n$. Suppose that any $o(n)$ vertices of $G_n$ meet
$o(n)$ edges of $G_n$, and that the same holds for $G_n'$.
Then $\dcG(G_n,G_n')\to 0$ if and only if $\dedit(G_n,G_n')\to 0$.
\end{lemma}
\begin{proof}
As noted above, $\dcG(G_n,G_n')\le 2\dedit(G_n,G_n')$, so our task is to show
that if $\dcG(G_n,G_n')\to 0$, then $\dedit(G_n,G_n')\to 0$.
Relabelling the vertices of $G_n'$, we may assume that
$\cn{\ka_{G_n}-\ka_{G_n'}}\to 0$.

Let $C$ be a constant such that $e(G_n), e(G_n')\le Cn$.
Suppose first that $G_n\nobreak\setminus\nobreak G_n'$
(or $G_n'\nobreak\setminus\nobreak G_n$) contains
a matching $M$ of size at least $cn$, for some constant
$c>0$. Set $a=c/(2C)>0$, and select
each edge of $M$ with probability $a$,
independently of the others. Write $M'$ for the set of edges
obtained, and $V$ for the vertex set of $M'$.
Then $\E(e(G_n[V]))$, 
the expected number of edges of $G_n$ spanned by $V$,
is at least $acn$, considering only edges in $M$. On the other
hand, if $e\notin M$ has both ends in the vertex set of $M$,
then the probability that both endpoints
of $e$ are included in $V$ is $a^2$,
so $\E(e(G_n'[V]))\le a^2e(G_n') \le a^2Cn$.
Recalling that $a=c/(2C)$, it follows that
\[
 \E\bb{e(G_n[V])- e(G_n'[V])} \ge acn-a^2Cn = c^2n/(4C) = \Theta(n).
\]
Choosing a set for which the random difference is at least its expectation,
we see that $\cn{\ka_{G_n}-\ka_{G_n'}}=\Theta(1)$, a contradiction.

From the above we may assume that the largest matchings in $G_n\setminus G_n'$
and in $G_n'\setminus G_n$ have size $o(n)$. But then there is
a set $V$ of $o(n)$ vertices that
meets every edge of the symmetric difference $G_n\diff G_n'$.
By assumption, $V$ meets only $o(n)$ edges of $G_n\cup G_n'$,
so it follows that $e(G_n\diff G_n')=o(n)$, so $\dedit(G_n,G_n')\to 0$.
\end{proof}

Note that Lemma~\ref{dcde} becomes false if the condition
that $o(n)$ vertices meet $o(n)$ edges is omitted, as shown by the example
mentioned earlier, where $G_n$ and $G_n'$ are two instances of the random
graph $G(\sqrt{n},1/2)$, each with $n-\sqrt{n}$ isolated vertices added.

Lemma~\ref{dcde} applies to two independent instances $G_1$, $G_2$
of the random graph $G(n,c/n)$.
Hence, $\dcG(G_1,G_2)\pto 0$ if and only if $\dedit(G_1,G_2)\pto 0$.
It is not too hard to see that the latter condition cannot hold
for any $c>1$.

\begin{theorem}\label{th_notclose}
For every $c>1$ there is a $\delta=\delta(c)>0$ such that,
if $G_1$ and $G_2$ are independent
instances of $G(n,c/n)$, then
whp the unnormalized edit distance between
$G_1$ and $G_2$ is at least $\delta n$.
\end{theorem}
In other words, normalizing with $p=1/n$, we have $\dedit(G_1,G_2)\ge \delta$
whp.
\begin{proof}
Let us start with an observation about $G(n,c/n)$.
Let $0<a<b$ be constants; we shall estimate the probability
of the event $E_\delta(H)$ that
$G_2=G(n,c/n)$ contains all but at most $\delta n$ edges
of some graph $H'$ isomorphic to $H$, where
$H$ is any given graph with $\floor{an}$ vertices and at least $bn$ edges,
and $\delta<b/2$.
There are $\binom{n}{\floor{an}}$ choices for the vertex set of $H'$,
and then at most $\floor{an}!$ graphs $H'$ with this vertex set
isomorphic to $H$. Finally, given $H'$, there are crudely
at most $\delta n \binom{e(H)}{\delta n}$ choices for the edges of $H'$
to omit, while the probability that $G(n,c/n)$ contains
the remaining edges is at most $(c/n)^{e(H)-\delta n}$. Hence,
\begin{eqnarray*}
 \Pr(E_\delta(H)) &\le&
 \binom{n}{\floor{an}} \floor{an}! \delta n \binom{e(H)}{\delta n}  (c/n)^{e(H)-\delta n} \\
 &\le& n^{an}  (e b/\delta)^{\delta n} (c/n)^{(bn-\delta n)}\delta n = n^{(a-b+\delta)n} e^{O(n)}.
\end{eqnarray*}
If $a$, $b$ and $\delta$ are constants with $\delta<b-a$, then the final probability is $o(1)$.

Turning to the proof of the theorem,
if $\dedit(G_1,G_2)\le\delta$, then the event $E_\delta(H)$ holds for any $H\subset G_1$.
If $c>2$, then $G_1$ itself has $n$ vertices and $cn/2+\op(n)$ edges,
so setting $a=1$ and $b=c'/2$ for any $2<c'<c$, whp $G_1$ has a subgraph
$H$ with $e(H)\ge bn$. Setting $\delta=(c'-2)/2>0$, whp we have
\[
  \Pr(\dedit(G_1,G_2)\le\delta \mid G_1) \le \Pr(E_\delta(H)) =o(1),
\]
from which the result follows.

If $1<c\le 2$, then the original results of Erd\H os and R\'enyi~\cite{ER60}
(see also~\cite{BBRG})
imply that there are functions $\rho(c)$ and $\zeta(c)$ with
$0<\rho(c)<\zeta(c)$ such that the largest component
of $G(n,c/n)$ has $\rho(c)n+\op(n)$ vertices and $\zeta(c)n+\op(n)$ edges.
Fixing $a<b$ with $\rho(c)<a<b<\zeta(c)$, it follows that whp $G_1=G(n,c/n)$
contains a subgraph $H$ with at most $an$ vertices and at least $bn$
edges. Taking $\delta=(b-a)/2>0$, the result follows as above.
\end{proof}

Using the results of Bollob\'as, Janson and Riordan~\cite{BJR},
the proof above may be extended easily to the much more general model $G_{1/n}(n,\ka)$,
although one first needs to decide what the appropriate statement is.
As in~\cite{BJR}, let $T_\ka$ be the integral operator associated to $\ka$,
defined by
\[
 (T_\ka f)(x) = \int_0^1 \ka(x,y) f(y) \dd\mu(y),
\]
and let $\norm{T_\ka}$ be its $L^2$-norm.
Roughly speaking, it was shown in~\cite{BJR}
that $G_{1/n}(n,\ka)$ has a giant component if and only if $\norm{T_\ka}>1$.
(There is a slight caveat here: the results of~\cite{BJR} assume
that $\ka$ is continuous almost everywhere; this assumption is only
needed due to the more general choice of the vertex types made there. It
is easy to see that these results apply to general $\ka$
if we choose the vertex types \iid, as we do in the definition of $G_{1/n}(n,\ka)$;
this is discussed in~\cite{BJRclust}.)

More precisely, it is shown in~\cite{BJR} that if $\ka$ satisfies a certain
irreducibility condition, then $G_{1/n}(n,\ka)$ has a `giant
component' with $\rho(\ka)n+\op(n)$ vertices and $\zeta(\ka)n+\op(n)$
edges, for constants $\rho(\ka)$ and $\zeta(\ka)$ satisfying
$0<\rho(\ka)<\zeta(\ka)$ whenever $\norm{T_\ka}>1$ (see Theorems 3.1 and 3.5
and Proposition 10.1 in~\cite{BJR}). Since any kernel effectively contains
an irreducible kernel (see Lemma 5.17 of~\cite{BJR}), it follows
that if $\ka$ is any kernel with $\norm{T_\ka}>1$, then there are constants
$0<a<b$ depending only on $\ka$ such that, whp,
$G_1=G_{1/n}(n,\ka)$ has a (connected) subgraph $H$ with at most $an$ vertices and at
least $bn$ edges. Setting $\delta=\delta(\ka)=(\zeta(\ka)-\rho(\ka))/2>0$,
the observation at the start of the proof of Theorem~\ref{th_notclose}
shows that, for any $c$, the probability that $G(n,c/n)$ contains
all but $\delta n$ edges of such a graph $H$ is $o(1)$.
If $\ka$ is bounded, then we may couple $G_2=G_{1/n}(n,\ka)$
and $G(n,\sup\ka /n)$ so that the former is a subgraph of the latter,
and it follows that whp independent copies $G_1$ and $G_2$ of $G_{1/n}(n,\ka)$
are at edit distance at least $\delta(\ka)/2$.

If $\ka$ is unbounded then we can approximate with bounded kernels as
in~\cite{BJR}; we omit the details, noting only that, considering
bounded kernels $\ka^M$ tending up to $\ka$, there is some $M_0$ such
that $\delta(\ka^{M_0})>0$. Then for any $M_1>M_0$, the argument above
shows that independent copies of $G_{1/n}(n,\ka)$ and $G_{1/n}(n,\ka^{M_1})$ are
whp at distance at least $\delta(\ka^{M_0})$, which does not depend on
$M_1$. This bound still applies if $M_1$ tends to infinity slowly enough,
but then $G_{1/n}(n,\ka^{M_1})$ and $G_{1/n}(n,\ka)$ are very close in $\dedit$.

As shown in~\cite[Proposition 8.11]{BJR},
the graph $G_{1/n}(n,\ka)$ satisfies the assumptions of Lemma~\ref{dcde}.
Putting the pieces together, we have thus
proved the following result.

\begin{theorem}
Let $\ka$ be a kernel with $\norm{T_\ka}>1$, and let
$G_n$ and $G_n'$ be independent instances of $G_{1/n}(n,\ka)$.
Then there is a constant $\delta>0$ such $\dedit(G_n,G_n')\ge \delta$
and $\dc(G_n,G_n')\ge \delta$ hold whp.\noproof
\end{theorem}

In fact, standard martingale arguments show that each of $\dedit(G_n,G_n')$
and $\dc(G_n,G_n')$ is concentrated about its mean, which
is thus bounded away from zero. We omit the details, which are
very similar to those in the proof of Theorem~\ref{dPconc} below.

As noted above for $G(n,c/n)$, the condition $\norm{T_\ka}>1$ is necessary:
otherwise, from the results of Bollob\'as, Janson and Riordan~\cite{BJR},
$G_{1/n}(n,\ka)$ consists almost entirely of small tree components,
with the number of copies of any given tree concentrated. It follows
easily that $\dedit(G_n,G_n')$ and hence $\dc(G_n,G_n')$ converge
to $0$ in probability (and almost surely) in this case.

\section{Tree counts}\label{treecounts}

Since it seems that we cannot do much with the cut metric
$\dc$ in the extremely sparse
case, let us now turn our attention to the count or subgraph
metric $\de$.
As in~\cite{BRsparse}, given two graphs $F$ and $G$
we write $\hom(F,G)$ for the number
of {\em homomorphisms} from $F$ to $G$, i.e., the number
of maps $\phi:V(F)\to V(G)$ such that $xy\in E(F)$ implies
$\phi(x)\phi(y)\in E(G)$,
and $\emb(F,G)$ for the number of {\em embeddings}
of $F$ into $G$, i.e., the number of injective homomorphisms
from $F$ to $G$.
If $G_n$ has $n$ vertices, we normalize the
homomorphism and subgraph counts by setting
\[
 t_p(F,G_n)= \frac{\hom(F,G_n)}{n^{|F|}p^{e(F)}} \quad\hbox{and}\quad
 s_p(F,G_n)= \frac{\emb(F,G_n)}{n_{(|F|)}p^{e(F)}},
\]
where $p=p(n)$ is our normalizing edge density, here $1/n$,
and $n_{(|F|)}=n(n-1)\cdots (n-|F|+1)$ is the number of possible embeddings of $F$.
In~\cite{BRsparse}, the subgraph metric $\de$ is defined
by choosing a certain set $\A$ of {\em admissible} graphs,
and defining $\de$ so that $(G_n)$ is Cauchy if and only
if $s_p(F,G_n)$ converges as $n\to\infty$ for each $F\in \A$.

Let $F$ be a connected graph which is not a tree.
The denominator in the definition of $t_p(F,G_n)$ or $s_p(F,G_n)$
is $\Theta(n^{|F|}p^{e(F)})$, which is order $1$ if $F$ is unicyclic,
and tends to zero if $F$ contains two or more cycles. This suggests that,
in this range, the parameters $s_p(F,\cdot)$ and $t_p(F,\cdot)$ make
sense only if $F$ is a tree, i.e., that we should take for $\A$
the set $\T$ of (isomorphism classes of) finite trees.
Indeed,
with $F$ unicyclic, convergence of $s_p(F,G_n)$ simply means
that for large $n$, every $G_n$ contains the same number of
copies of $F$. This condition is very far from the kind of 
global graph property we are looking for.
Since the expected number of copies of a connected graph
$F$ in $G(n,c/n)$ tends to infinity if and only if $F$ is tree,
roughly speaking we do not expect to see small cycles in graphs
with $\Theta(n)$ edges. Of course, there are natural examples
of extremely sparse graphs containing many short cycles, but we should
handle these differently; see Section~\ref{ss_nontree}.
For now, we shall consider graphs that, like $G(n,c/n)$,
contain few short cycles.
More formally, throughout this section we assume
that $(G_n)$ is {\em asymptotically treelike}, in the sense
that
\begin{equation}\label{at}
 \emb(F,G_n) = o(n)
\end{equation}
for any connected $F$ that is not a tree. Under a suitable
assumption on the degrees in $G_n$, it
suffices to impose condition \eqref{at} for cycles.

Under the assumption \eqref{at}, it is easy to see that the parameters
$(s_p(T,G_n))_{T\in \T}$ and $(t_p(T,G_n))_{T\in \T}$ are essentially
equivalent. In particular, up to a $o(1)$ error, for any tree $T$,
$t_p(T,G_n)$ can be written as a linear
combination of the parameters $s_p(T',G_n)$, $|T'|\le |T|$, and vice versa.
We shall work with $s_p(T,G_n)$, which is more natural. Adjusting
the normalizing constant very slightly, we shall simply set
\[
 s_p(T,G_n) = \emb(T,G_n)/n.
\]
As in~\cite{BRsparse}, we assume that the normalized
counts of all admissible subgraphs remain bounded.
In other words,
we shall assume that
\[
 \sup_n s_p(T,G_n) \le c_T <\infty
\]
for each
tree $T$. In fact, it will be convenient to make the stronger assumption
that the tree
counts are {\em exponentially bounded}, i.e., 
that there
is a constant $C$ such that
\[
 \limsup s_p(T,G_n)\le C^{e(T)}
\]
for every tree $T$.
For example, taking $T$ to be a star, this condition
implies that the $k$th moment of the degree of a random vertex of $G_n$
is at most $C^k+o(1)$  as $n\to\infty$. As in~\cite{BRsparse},
writing $\F$ for the set of isomorphism classes of finite graphs,
and enumerating the set
$\T$ of isomorphism classes of finite trees
as $T_1,T_2,\ldots$, define a map
\[
  s_p:\F_0\to X,
  \qquad  G\mapsto (s_p(T_i,G))_{i=1}^\infty,
\]
where $\F_0$ is the subset of $\F$ consisting of graphs satisfying the
tree-count bounds above, and $X=\prod_i [0,c_{T_i}]$.
We then define the (tree) subgraph
distance $\de(G_1,G_2)$ between
two graphs $G_1$ and $G_2$ as $d(s_p(G_1),s_p(G_2))$, where $d$ is any metric
on $X$ giving rise to the product topology.
(It is not clear whether $\de$ is a metric
or only a pseudometric, but this is not important.)
Defining 
\[
 s(F,\ka) = \int_{[0,1]^k} \prod_{ij\in E(F)}\ka(x_i,x_j)\prod_{i=1}^k \dd x_i,
\]
as before, we may extend $s_p$ to bounded kernels $\ka$
by mapping $\ka$ to $(s(T_i,\ka))_{i=1}^\infty$, and it then
makes sense to consider
the question of when $\de(G_n,\ka)\to 0$.

Unlike the cut metric, it is certainly possible to have convergence in the metric
$\de$. Indeed, it is easy to check that
$s_p(T,G_{1/n}(n,\ka))$ converges in probability to $s(T,\ka)$
for any tree $T$ and any bounded kernel $\ka$, and one can then show that
$\de(G_{1/n}(n,\ka),\ka)\to 0$ in probability and (coupling suitably)
with probability $1$.

In this section, the main questions we shall
consider are: which points $X$ of $[0,1]^\T$
are realizable as limits of sequences 
$(s_p(G_n))$, where $(G_n)$ is asymptotically
treelike and has bounded tree counts,
and how do these limit points relate to kernels?
In fact, we shall reformulate these questions slightly.

\medskip
Recall that in the definition of $\emb(T,G_n)$, the tree $T$ is labelled.
Let us also regard $T$ as rooted, with root vertex $1$, say.
Letting $v$ denote a vertex of $G_n$ chosen uniformly at random,
$s_p(T,G_n)$ is simply the expected number of embeddings from $T$ 
into $G_n$ mapping the root to $v$. Letting $T$ run over stars,
the numbers $s_p(T,G_n)$ give the moments of the degree of $v$;
assuming these moments converge, and that the limiting moments
are exponentially bounded, a standard result implies that 
they determine the limiting degree distribution.
(For an example
of two non-negative integer valued distributions with the
same finite moments, see Janson~\cite{Jrounding}, taking,
for example, $\alpha=\log(2/3)$ and $\alpha=\log(4/5)$ in
Example 2.12.) More generally, the parameters $s_p(T,G_n)$
for all finite trees $T$ provide a sort of moment characterization
of the local neighbourhood of $v$.
Presumably, if these tree
counts are exponentially bounded, then they characterize
the distribution of the local neighbourhood of $v$.
(It should not be hard to adapt the proof that convergence
of exponentially bounded moments implies convergence in distribution
based on the Jordan--Bonferroni inequalities:
consider the event that certain edges are present, forming
a copy of some given tree $T$ with $v$ as the root, and apply inclusion--exclusion
to calculate the probability that certain other edges are
absent, so $T$ is the local neighbourhood of $v$; we have
not checked the details.)
In fact,
for this something weaker than exponentially bounded tree counts
should be enough, but exponential boundedness is a natural
assumption, as it is the global analogue of the (perhaps too restrictive)
local condition that all degrees are uniformly bounded.
In any case, it makes more sense to study the distribution of local neighbourhoods
directly, rather than its moments.

For $t\ge 0$, let $\Ga_{\le t}(v)$ denote the subgraph of $G_n$
formed by all vertices within graph distance $t$ of $v$,
regarded as a rooted graph with root $v$.
Let $\Tr_t$ denote the set of isomorphism classes of rooted trees
of height at most $t$, i.e., such that every vertex is within distance
$t$ of the root. Finally, for $T\in \Tr_t$, let
\[
 p(T,G_n) = p_t(T,G_n) = \Pr\bb{ \Ga_{\le t}(v) \isom T},
\]
where $v$ is a vertex of $G_n$ chosen uniformly at random,
and $\isom$ denotes isomorphism as rooted graphs.
Officially, the subscript $t$ is necessary in the notation,
since if $T\in \Tr_t\subset \Tr_{t+1}$, then $\Ga_{\le t}(v)\isom T$
and $\Ga_{\le t+1}(v)\isom T$ are different conditions.
In practice, we work with the disjoint union $\Tr$ of
the sets $\Tr_t$, $t=0,1,2,\ldots$,
so each $T\in \Tr$ comes with a height $t$ which we do not indicate
in the notation; $T$ may or may not contain vertices at distance $t$ from the root.

Since each $p(T,G_n)$ lies in $[0,1]$, any sequence $(G_n)$
has a subsequence on which $p(T,G_n)$ converges for every $T$.
We would like to study the set $\P$ of possible limits $(p(T))_{T\in \Tr}$
arising from asymptotically treelike sequences with exponentially bounded
tree counts.

We start with some simple observations. First note
that if $(G_n)$ is asymptotically treelike, then
\[
 \sum_{T\in \Tr_t} p(T,G_n) \to 1
\]
as $n\to\infty$ with $t$ fixed; this is simply the statement
that only $o(n)$ vertices have cycles in their $t$-neighbourhoods.
In general, such a statement does not carry over to the limit,
since infinite sums are not continuous with respect to pointwise
convergence. Recall that we are assuming that
the tree counts in $(G_n)$ are (exponentially) bounded.
It is easy to check that, under this assumption,
for any $t\ge 0$ and any $\eps>0$, there exists an $M$
such that, for large enough $n$, at most $\eps n$ vertices
$v$ have more than $M$ edges in their $t$-neighbourhoods.
(In the terminology of probability theory, for each $t$
the sequence $(X_{n,t})$ of random variables given by
$X_{n,t}=e(\Ga_{\le t}(v))$ is `tight', where $v$ is a uniformly
chosen random vertex of $G_n$. Equivalently,
the random variables $\Ga_{\le t}(v)$, $n=1,2,\ldots$, are themselves tight.)
In other words, for large $n$,
at least $1-\eps$ of the mass of the distribution
$(p(T,G_n))_{T\in \Tr_t}$ is concentrated on a fixed, finite
set of trees. Under this condition, pointwise convergence
implies convergence in $L^1$, and it follows
that $\sum_{T\in \Tr_t} p(T)=1$ whenever $(p(T))\in \P$.
In other words, $p$ gives a probability distribution on $\Tr_t$
for each $t$.

Given a finite or infinite rooted tree $T$, we may define the {\em restriction}
of $T$ to height $t$ to be the subtree $T|_t$ of $T$ induced by all vertices
within distance $t$ of the root, which we naturally regard as an
element of $\Tr_t$. If $(G_n)$ is asymptotically treelike, then,
for any $T\in \Tr_t$, we have
\[
 p(T,G_n) = \sum_{T'\in \Tr_{t+1}: T'|_t=T} p(T',G_n) +o(1).
\]
(The $o(1)$ correction appears because of the possibility that
the $t+1$ neighbourhood of a random vertex $v$ contains a cycle
while the $t$ neighbourhood does not.)
Using $L^1$ convergence, it follows that
\[
 p(T) = \sum_{T'\in \Tr_{t+1}: T'|_t=T} p(T')
\]
whenever $(p(T))\in \P$. In other words, the distribution on $\Tr_t$ is
obtained from that on $\Tr_{t+1}$ by the natural restriction operation.

This last fact allows us to combine the distributions on $\Tr_t$ given
by $p\in \P$ into a single probability distribution on the
set $\Tr_\infty$ of (finite or infinite) locally finite rooted trees.
Of course, we take for the set of measurable events the $\sigma$-field
generated by the sets
\[
 E_{t,T} = \{T'\in \Tr_\infty: T'|_t = T\}
\]
for $t=0,1,2,\ldots$ and $T\in \Tr_t$.
We say that a probability distribution $\pi$ on $\Tr_\infty$
is the {\em local limit} of a sequence $(G_n)$ of graphs
with $|G_n|=n$ if
\[
 \lim_{n\to\infty} p_t(T,G_n) = \pi(E_{t,T})
\]
holds for every $t\ge 0$ and every $T\in \Tr_t$.
We should like to know which probability distributions $\pi$ on $\Tr_\infty$
arise as local limits of asymptotically treelike, exponentially
bounded sequences $(G_n)$.
(For a closely related question of Aldous and Lyons~\cite{AL},
see Section~\ref{ss_nontree}.)

We shall give some examples of such distributions in the next section,
arising from random graphs. Here, we give a trivial example:
fix $d\ge 2$ and let $\pi$ be the distribution that is concentrated on 
one point of $\Tr_\infty$, the (infinite) $d$-regular tree. This distribution
arises as a local limit if we take $G_n$ to be a random $d$-regular
graph with $n$ vertices.

Having given a trivial example of a distribution that can arise, let
us give a trivial example of one that cannot. Let $\pi$ be any distribution
concentrated on the two trees $T\in \Tr_\infty$ in which every vertex has
degree $2$ or $3$, and the neighbours of a vertex of degree $2$
have degree $3$ and vice versa. (There are two such trees
since our trees are rooted.) Considering graphs $G_n$ in which the
vertices have degrees $2$ and $3$, there is no `local reason' why the
neighbourhoods of a random vertex $v$ shouldn't look like the first
few generations of $\pi$. However, unless $\pi$ satisfies
an extra condition, there is of course a global reason:
there are as many edges from vertices of degree $2$ to vertices
of degree $3$ as vice versa. Thus,
if all edges join vertices of different degrees, we must have $3/2$
times as many degree $2$ vertices as degree $3$ vertices,
and if $\pi$ is to arise as a local limit, it must assign probabilities $3/5$
and $2/5$ to the trees in which the root has degree $2$ and $3$,
respectively.

More generally, consider the following two ways of picking
a (not uniformly) random vertex of $G_n$. (A) pick a vertex
$v$ with probability proportional to its degree.
(B) pick a vertex $w$ with probability proportional to its degree,
then choose an edge incident with $w$ uniformly,
and let $v$ be the other end of this edge.
It is easy to see that (A) and (B) give the same distribution
for the vertex $v$ - indeed, we are simply choosing an edge
$e$ of $G$ at random, and then picking an end of $e$ at random.
In (B) we `change our minds' after picking the random end, which makes no difference.
The equivalence of (A) and (B) gives rise to a consistency condition
on our distributions $\pi$.

Let $\pi$ be any probability distribution on $\Tr_\infty$ in which
the expected degree of the root is finite. (Any distribution
in $\P$ will have this property,
due to the bounded average degree of graphs in $G_n$.)
Then we may define the `root-sized-biased' version $\tpi$
of $\pi$ to be the distribution $\pi$ biased by the degree
of the root. In other words, for any $t\ge 1$ and $T\in \Tr_t$,
\[
 \tpi(E_{t,T}) = \frac{d_0(T) \pi(E_{t,T}) }{\E_\pi(d_0)},
\]
where the $d_0(T)$ denotes the degree of the root of $T$,
and $\E_\pi(d_0)$ is simply the expectation of the degree
of the root of a $\pi$-random $T\in \Tr_\infty$.
Let us define the {\em shift} $\tpi^*$ of $\tpi$ to be the distribution
on $\Tr_\infty$ obtained as follows: first choose a $T\in \Tr_\infty$
according to the distribution $\tpi$. Then pick a neighbour
$v$ of the root uniformly from among the $d_0(T)$ neighbours.
Finally, let $T'\in\Tr_\infty$ be the tree obtained from $T$
by taking $v$ as the root. Since the restriction of $T'$
to height $t$ is determined by the restriction of $T$ to height $t+1$,
it is easy to check that this does define a probability
distribution $\tpi^*$ on $\Tr_\infty$.

Using the equivalence of the procedures (A) and (B) above for picking a
random vertex of $G_n$, it is easy to see that if $\pi$ arises
as the local limit of one of our sequences $(G_n)$, then $\pi$
is {\em shift invariant}, in that $\tpi^*=\tpi$.
It is tempting to believe that this condition is sufficient,
but in fact, as pointed out to us by G\'abor Elek,
this is not the case, as we shall now explain.

An infinite graph is called {\em quasi-transitive} if the action of
its automorphism group on the vertex set induces a finite number of orbits,
i.e., if there are only finitely many different `types' of vertices
in the graph. A quasi-transitive tree may be described
by a square matrix $A=(a_{ij})$ specifying, for each $i$ and $j$,
the number of type-$j$ neighbours each vertex of type $i$ has.
Also, given any square matrix $A$ with non-negative integer entries
in which $a_{ij}>0$ if and only if $a_{ji}>0$, one can construct
a corresponding quasi-transitive tree. (This correspondence
is not one-to-one; it may be that vertices corresponding to different
rows of $A$ end up having the same type. For example, if each
row of $A$ has the same sum $d$, then $T$ is simply the $d$-regular
tree. It is easy to describe conditions on $A$ under which
this kind of `collapse' does not happen.)

\begin{example}\label{eT234} {\bf A non-unimodular tree.}
Let $T$ be the infinite (unrooted) tree corresponding to the matrix
\[
 A= \left(\begin{matrix}
  0 & 1 & 1 \\ 
  2 & 0 & 1 \\
  1 & 3 & 0
\end{matrix}\right).
\]
Thus vertices in $T$ have degree 2, 3 or 4, each vertex has
one neighbour of the `next' degree (where $2$ follows $4$),
and 1, 2 or 3 neighbours of the previous degree.
There are three rooted trees corresponding to $T$; let us call
these $T_2$, $T_3$ and $T_4$, where the root of $T_i$ has degree $i$.

Suppose that $\pi$ is a probability distribution on $\Tr_\infty$
supported on $\{T_2,T_3,T_4\}$, and giving mass $\pi_i$
to $T_i$. Suppose also that $\pi$ is the local limit
of an  asymptotically
treelike sequence $(G_n)$. Then, considering the $1$-neighbourhood
of a random vertex, the convergence assumption
implies that $G_n$ has $\pi_i n+o(n)$ vertices of each degree $i$, $i=2,3,4$.
Considering the 2-neighbourhood, we also see that $\pi_2 n+o(n)$
vertices of degree $2$, i.e., all but $o(n)$ vertices of degree $2$,
have one neighbour of degree $3$ and one of degree $4$.
Hence the number of edges of $G_n$ from degree 2 vertices to degree 3
vertices is $\pi_2 n+o(n)$. But, similarly, all but $o(n)$
of the $\pi_3 n+o(n)$ vertices of degree 3 have two neighbours of degree
2, so there are $2\pi_3 n+o(n)$ edges from degree 3 vertices to degree 2 vertices.
Hence, $\pi_2=2\pi_3$. Similar arguments show that $\pi_3=3\pi_4$ and $\pi_4=\pi_2$.
Together these equations imply that $\pi_i=0$ for all $i$, contradicting $\pi_2+\pi_3+\pi_4=1$.
Thus there is no probability distribution supported on $\{T_2,T_3,T_4\}$
that is a local limit.

A little calculation shows that taking
$\pi_2=9/20$, $\pi_3=7/20$ and $\pi_4=4/20$
gives a shift-invariant distribution supported on $\{T_2,T_3,T_4\}$,
so this shift-invariant distribution is not a local limit.
\end{example}

The reason for the terminology `non-unimodular' above
will become clear in Section~\ref{ss_nontree}.
A different example of a non-unimodular tree is given in Example~3.1
of Benjamini, Lyons, Peres and Schramm~\cite{BLPS_GAFA99}, corresponding
to the matrix
\[
 A= \left(\begin{matrix}
  0 & 1 & 1 \\ 
  1 & 0 & 2 \\
  2 & 1 & 0
\end{matrix}\right).
\]

The argument leading to $\pi_2=2\pi_3$ in the example above can be generalized.
In this argument, we drew an oriented edge from a vertex $x$ to a neighbour $y$
if the local neighbourhood (in this case the $2$-neighbourhood) 
of $x$ satisfied a certain rule (that $x$ had degree $2$
and $y$ degree $3$). We can of course use any such rule.
For this it turns out to be useful to consider {\em doubly rooted trees},
i.e., trees with an ordered pair of adjacent distinguished vertices
$x$ and $y$. Formally, we shall consider $\Trr_\infty$,
the set of triples $\{(T,x,y)\}$ in which
$(T,x)$ is a rooted locally finite tree and $y$ is a neighbour of $x$.
(We do not quotient by isomorphisms at this stage.)
Similarly, we consider the set $\Trr_t=\{(T,x,y)\}$ of finite doubly rooted
trees in which all vertices are within distance $t$ of $x$.

Let $A\subset \Trr_t$ be any {\em isomorphism invariant} set
of doubly rooted trees of radius $t$.
We say that a (potentially) infinite doubly rooted tree $(T,x,y)$ is in
$A$ if its restriction to the $t$-neighbourhood of the root is in $A$.
Note that if $x$ and $y$ are adjacent vertices of $T$, then $(T,y,x)$
is to be viewed as a tree rooted at $y$, with a second distinguished vertex $x$.
We claim that if $\pi$ is a local limit, then
\begin{equation}\label{io1}
 \E_\pi |\{y: (T,x,y)\in A\}| = \E_\pi |\{y: (T,y,x)\in A\}|,
\end{equation}
where the expectation is over the choice of a random rooted tree $(T,x)$
with distribution $\pi$. The argument is as above so let us just outline it:
let $(G_n)$ be a sequence of finite graphs converging to $\pi$
in the appropriate sense. In each $G_n$, draw a directed edge from a vertex $x$
to a neighbour $y$ if and only if $(\Ga_{\le t}(x),x,y)\in A$, where
$\Ga_{\le t}(x)$ is the $t$-neighbourhood of $x$ in $G_n$. 
Now the limiting fraction of vertices of $G_n$ whose $t$-neighbourhood
has a certain form is given by $\pi$. It follows that the expected out-degree
of a random vertex of $G_n$ converges to the left-hand side of \eqref{io1}.
On the other hand, the limiting fraction of vertices of $G_n$ whose $(t+1)$-neighbourhood
has a certain form is again given by $\pi$. From the $(t+1)$-neighbourhood
of $x$ one can obtain the $t$-neighbourhood of each neighbour $y$ of $x$,
and thus decide whether we drew an edge from $y$ to $x$. It follows
that the expected in-degree converges to the right-hand side of \eqref{io1}.
Since in any finite directed graph, the average out- and in-degrees are equal,
\eqref{io1} follows.

Let $\isom$ denote isomorphism of doubly-rooted trees. In the above, we took
$A$ to be any subset of $\Trr_\infty/\isom$ such that whether $(T,x,y)\in A$ holds
is determined by the $t$-neighbourhood of $x$ for some $t=t(A)$. In general,
we can consider arbitrary subsets of $\Trr_\infty/\isom$ that may
be approximated by such subsets, i.e., arbitrary measurable subsets
of $\Trr_\infty/\isom$. Also, replacing a subset by its characteristic function,
there is no need to restrict to $0$-$1$ functions. This leads
us to the following definition.

Let $\pi$ be a probability measure on $\Tr_\infty$. Then $\pi$ is called
{\em involution invariant} if for every non-negative measurable function
$f$ on $\Trr_\infty/\isom$ we have
\begin{equation}\label{io2}
 \E_\pi \sum_y f(T,x,y) = \E_\pi \sum_y f(T,y,x)
\end{equation}
where the expectation is over the $\pi$-random rooted tree $(T,x)$, and the
sum is over all neighbours $y$ of $x$. Note that $f$ must be isomorphism
invariant, but if the root $x$ of $T$ has degree $d$, then there are $d$ terms in the sums above,
even if several of these correspond to isomorphic doubly-rooted trees.
Note also that it suffices to consider functions $f$ that are characteristic
functions of measurable sets.

We have seen above that if $\pi$ is a local limit then $\pi$ must be involution invariant.
This observation was first made (in a slightly different context)
by Benjamini and Schramm~\cite{BSrec}; we return to this in Section~\ref{ss_nontree}.
We do not know whether this necessary condition on $\pi$
is sufficient. (See also Question~\ref{qglim}.)

\begin{question}\label{qtlim}
Does every involution-invariant probability distribution on $\Tr_\infty$
arise as the local limit of some sequence $(G_n)$ of graphs
with $|G_n|=n$?
\end{question}
The sequence $(G_n)$ above will necessarily be asymptotically treelike
(otherwise the total weight of $\pi$ would be less than $1$,
so $\pi$ would not be a probability distribution). However,
in the question above we have lost the condition that the tree
counts of $(G_n)$ be exponentially bounded. Such a condition
may or may not be needed to get sensible limiting behaviour.
To avoid possible complications, in the first draft of this
paper we posed the following variant of Question~\ref{qtlim}.

\begin{question}\label{qtlimD}
Let $\Delta\ge 2$ be given, and let $\pi$ be an involution-invariant
probability distribution on the set of trees $T\in \Tr_\infty$
with maximum degree at most $\Delta$.
Must $\pi$ arise as the local limit of some sequence $(G_n)$
of graphs with all degrees at most $\Delta$?
\end{question}

Question~\ref{qtlimD} has now been answered in the affirmative
by Elek~\cite{Elek}.

Let us finish this section by noting that there is a certain
class of distributions for which the answer to the questions
above is trivially yes, namely the class of 
probability distributions $\pi$ on {\em finite} rooted trees.
It is easy to check that a distribution $\pi$ on the set of finite rooted
trees (a subset of $\Tr_\infty$) corresponds in the natural way
to a distribution on unrooted trees if and only if $\pi$
is involution invariant. In this case it is easy to construct
forests $G_n$ with the appropriate limiting distribution of trees.

\section{Tree counts in random graphs}

We have seen in the previous section that the metric $\de$, defined
using only tree counts, makes sense in the extremely sparse case ($p=1/n$),
in that there are non-trivial convergent sequences: unlike the cut metric,
$\de$ is not `too strong' to be meaningful. Here we ask whether it is `too weak'
to be useful,
i.e., whether too many sequences converge that `should not'. Of course
this is rather a vague question, so we concentrate on a precise special
case: we have seen that for any bounded kernel we have
$\de(G_{1/n}(n,\ka),\ka)\to 0$ with probability 1, so we ask: which kernels
are distinguished by $\de$? In other words, when do two kernels
$\ka_1$ and $\ka_2$ satisfy $s(T,\ka_1)=s(T,\ka_2)$ for every tree $T$?
As in the previous section, we shall change viewpoint slightly,
considering the distribution of local neighbourhoods, rather than counting
subgraphs.

Adopting the terminology of Bollob\'as, Janson and
Riordan~\cite{BJR}, let $(\sss,\mu)$ be an arbitrary probability
space. By a {\em kernel} on $(\sss,\mu)$ we mean an integrable, symmetric, non-negative
function on $\sss\times \sss$. So far we have almost always
taken $\sss=[0,1]$ and $\mu$ Lebesgue measure, but the notation is
more convenient if we are rather more general here. As in~\cite{BJR}
(but taking the special case where the vertex types are \iid),
suppressing the dependence on $(\sss,\mu)$ in the notation, we may
form a random graph $G_{1/n}(n,\ka)$ as follows: let $x_1,\ldots,x_n\in
\sss$ be \iid\ with the distribution $\mu$, and then, given
$(x_1,\ldots,x_n)$, join each pair of vertices $\{i,j\}\subset
[n]^{(2)}$ with probability $\min\{\ka(x_i,x_j)/n,1\}$, independently of the other
pairs. We say that vertex $i$ has {\em type} $x_i$ and call
$(\sss,\mu)$ the {\em type space}.

Let $\bp_\ka$ be the multi-type Poisson Galton--Watson branching process
naturally associated to $\ka$: we start in generation $0$ with a
single particle whose type is distributed according to $\mu$. A
particle of type $x$ has children whose types form a Poisson process
on $\sss$ with the distribution $\ka(x,y)\dd\mu(y)$: the number of
such children in a measurable set $A\subset \sss$ is Poisson with
mean $\int_A \ka(x,y)\dd\mu(y)$. As usual, the children of different
particles are independent, and independent of the history. This
branching process is the key to the analysis of the random graph
$G_{1/n}(n,\ka)$ in~\cite{BJR}.

The branching process $\bp_\ka$ is of course simply a random rooted tree
with labels attached to the vertices, giving the type of each vertex.
Forgetting the labels, we may regard $\bp_\ka$ as a random rooted
tree; we write $\pi_\ka$ for the corresponding probability measure
on $\Tr_\infty$. It is not hard to check
that if $\ka$ is bounded, then $\pi_\ka$ provides the correct local approximation
for $G_{1/n}(n,\ka)$: in the notation of the previous section,
for each $T\in \Tr_t$ we have
\[
 p_t(T,G_{1/n}(n,\ka)) \to  \pi_\ka(E_{t,T}) \hbox{ with probability }1.
\]
This is the distributional equivalent of the convergence
in moments given by $s_p(T,G_{1/n}(n,\ka))\to s(T,\ka)$ for every tree $T$.

In the light of the comments above, we should like to answer the following
question: when do two different branching processes $\bp_{\ka_1}$
and $\bp_{\ka_2}$ give rise to the same random tree?
In other words,
when is $\pi_{\ka_1}=\pi_{\ka_2}$? It is not hard to check that,
at least for bounded $\ka$, the counts $(s(T,\ka))_{T\in \T}$
determine $\pi_\ka$ and vice versa, so this is the same question as that asked
at the start of the section. Since $\pi_{\ka}$ directly describes the local structure
of $G_{1/n}(n,\ka)$, we consider the present branching process
formulation more informative.

There is an obvious case when $\pi_{\ka_1}=\pi_{\ka_2}$: let $\ka_i$
be a kernel on $(\sss_i,\mu_i)$.
We say that $\ka_1$ {\em refines} $\ka_2$, and write $\ka_1\prec \ka_2$,
if there is a measure-preserving map $\tau:\sss_1\to\sss_2$
such that for $\mu_1$-almost every $x\in \sss_1$ we have
\[
 \int_{\tau^{-1}(A)} \ka_1(x,y) \dd\mu_1(y) = \int_A \ka_2(\tau(x),y) \dd\mu_2(y)
\]
for all measurable $A\subset \sss_2$.
(This is a very different notion to that appearing in \ssequiv,
despite the superficial similarity to $\ka_1=\ka_2^{(\tau)}$.)
In other words, if we take a
particle of $\bp_{\ka_1}$ and look at the distribution of the images
under $\tau$ of the types of its children, then this distribution
depends only on the image of the type of the original particle, and
it does so according to the kernel $\ka_2$. From this description it
is immediate that if $\ka_1\prec \ka_2$, then
$\pi_{\ka_1}=\pi_{\ka_2}$.

From now on we shall concentrate on the finite-type case, i.e., take $\sss$ to be finite.
Note that there is a natural correspondence between this case
and the case of kernels $\ka$ on $[0,1]^2$ that are piecewise constant on rectangles.
In this case $\ka_1\prec\ka_2$ simply means that the types associated to $\ka_1$
may be grouped together to form the types associated to $\ka_2$,
and the distribution of the grouped types of the children of a particle in $\bp_{\ka_1}$
is what it should be in $\bp_{\ka_2}$.

The relation $\prec$ is clearly transitive. Hence the natural
conjecture is that two kernels give the same distribution on trees
if and only if they have a common refinement. Or should it be if and only
if they are both refinements of a common `coarsening'? In fact,
somewhat surprisingly, the two are equivalent!

\begin{theorem}\label{updown}
Let $\ka_1$ and $\ka_2$ be finite-type kernels. Then the following are equivalent:
(i) there is a finite-type kernel $\kar$ with $\kar\prec\ka_1$ and $\kar\prec\ka_2$,
and 
(ii) there is a finite-type kernel $\kac$ with $\ka_1\prec\kac$ and $\ka_2\prec\kac$.
\end{theorem}

\begin{proof}
Let $\ka_i$ have type-space $(\sss_i,\mu_i)$.
Since the definition of $\prec$ ignores sets of measure zero, we may assume
that each $\mu_i$ is a strictly positive measure on the finite set $\sss_i$.

We start by showing that (ii) implies (i).
Suppose then that $\ka_1\prec\kac$ and $\ka_2\prec\kac$. Let $\kac$ have
type space $(\sss_\cc,\mu_\cc)$. Then each of $\bp_{\ka_1}$, $\bp_{\ka_2}$ may
be viewed as $\bp_{\kac}$ with `extra labels' on the vertices,
indicating the subtypes in $\sss_i$. We wish to add labels of both kinds simultaneously.
It is easy to write down a way of doing so.

Let $\tau_i:\sss_i\to \sss_\cc$ be the map witnessing $\ka_i\prec\kac$.
Let $\sss_\rr$ be the set of pairs $(i,j)\in \sss_1\times\sss_2$ with
$\tau_1(i)=\tau_2(j)$, and set
\[
 \mu_\rr((i,j)) = \frac{\mu_1(i)\mu_2(j)}{\mu_\cc(\tau_1(i))}.
\]
Summing first over all $i\in \sss_1$ and $j\in \sss_2$ mapping
to a given $k\in \sss_c$, it is easy to check that $\mu_\rr$
is a probability measure on the finite set $\sss_\rr$. It remains
to define an appropriate kernel.

For $(i,j)$, $(k,\ell)\in \sss_\rr$, set
\[
 \ka_\rr((i,j),(k,\ell)) = \frac{\ka_1(i,k) \ka_2(j,\ell)}{\kac(\tau_1(i),\tau_1(k))},
\]
whenever the denominator is non-zero, and set $\ka_\rr((i,j),(k,\ell))=0$ otherwise.

Since, despite appearances, the definitions above are symmetric with respect
to $\ka_1$ and $\ka_2$, to establish (i) it suffices to show that $\ka_\rr\prec\ka_1$.
Of course, in doing so we shall map $(i,j)\in \sss_\rr$ to $i\in \sss_1$.
This map is measure preserving.
Since $\ka_\rr$ is of finite type, all we must check
is that for any $(i,j)\in \sss_\rr$ and $k\in \sss_1$,
a particle of type $(i,j)$ in $\bp_{\sss_\rr}$ has the right number of children
of (subtypes of) type $k$. In other words, we must show that
\[
 \sum_{\ell\::\:\tau_2(\ell)=\tau_1(k)} \ka_\rr((i,j),(k,\ell))\mu_\rr((k,\ell)) = \ka_1(i,k) \mu_1(k).
\]
If $\kac(\tau_1(i),\tau_1(k))$ is zero then $\ka_1(i,k)$ is also zero (since $\ka_1\prec \kac$),
so both sides are zero. Otherwise,
the left hand side above is simply
\[
 \sum_{\ell\::\:\tau_2(\ell)=\tau_1(k)} \frac{\ka_1(i,k) \ka_2(j,\ell)}{\kac(\tau_1(i),\tau_1(k))}
  \frac{\mu_1(k)\mu_2(\ell)}{\mu_\cc(\tau_1(k))}
 = \ka_1(i,k)\mu_1(k) \frac{\sum_{\ell\::\:\tau_2(\ell)=\tau_1(k)} \ka_2(j,\ell)\mu_2(\ell)}
 { \kac(\tau_1(i),\tau_1(k))\mu_\cc(\tau_1(k))}.
\]
Recalling that $\tau_1(i)=\tau_2(j)$, the final fraction is $1$ since $\tau_2$
is a map witnessing $\ka_2\prec \kac$: the numerator and denominator
both express, for a particle of type $j\in\sss_2$ (and hence
of type $\tau_2(j)\in \sss_\cc$), the expected number of children
of type $\tau_1(k)$. This completes the proof that (ii) implies
(i).

Suppose now that (i) holds, i.e., that $\ka_1$ and $\ka_2$ have a common refinement $\kar$.
Each type in the corresponding type space maps to some $i\in\sss_1$ and some $j\in\sss_2$.
It is easy to see that we may group together all types mapping to the same pair $(i,j)$,
obtaining a common refinement of $\ka_1$ and $\ka_2$ that we also
denote $\kar$, with type space $\sss_\rr$ a subset of $\sss_1\times\sss_2$.
As before, we delete any types with measure zero.
We may regard $\sss_\rr$ as the edge set of a bipartite graph $G$ with vertex classes
$\sss_1$ and $\sss_2$; we shall construct the common coarsening $\kac$ as a kernel
with type space the set of components of $G$.
Since $\ka_\rr\prec \ka_1$, the map from $\sss_\rr$ to $\sss_1$ given by $(i,j)\mapsto i$
is measure preserving. In other words, for each $j$, writing $ij$ for $(i,j)$,
\begin{equation}\label{sij}
 \sum_{j\::\:ij\in E(G)} \mu_\rr(ij) = \mu_1(i).
\end{equation}
Similarly,  $\sum_{i\::\:ij\in E(G)} \mu_\rr(ij) = \mu_2(j)$.
For each component $C$ of $G$, set $\mu_\cc(C)=\sum_{ij\in E(C)} \mu_\rr(ij)$;
this defines a probability
measure on the set $\sss_\cc$ of components of $G$.
If $C$ has vertex set $C_1\cup C_2$, $C_i\subset \sss_i$, then from \eqref{sij},
\[
 \mu_\cc(C) = \sum_{i\in C_1}\sum_{j\::\:ij\in E(G)} \mu_\rr(ij) = \sum_{i\in C_1} \mu_1(i).
\]
Hence the map $\tau_1:\sss_1\to\sss_\cc$ mapping each $i\in \sss_1\subset V(G)$
to the component in which it lies is measure preserving. Similarly for
the corresponding map $\tau_2:\sss_2\to\sss_\cc$.

Fix two components $C$ and $C'$ of $G$, which need not be distinct.
For each edge $e\in C$ set
\[
 \la(e,C') = \sum_{f\in E(C')} \ka_\rr(e,f)\mu_\rr(f),
\]
the expected number of children in $C'$ of a particle in $\bp_{\ka_\rr}$ of type $e$.
Writing $e=ij$, we may rewrite $\la(e,C')$ as follows:
\[
 \la(ij,C') = \sum_{k\in C_1'}\sum_{k\ell\in E(G)} \ka_\rr(ij,k\ell)\mu_\rr(kl)
 = \sum_{k\in C_1'} \ka_1(i,k)\mu_1(k),
\]
where the second step is from $\ka_\rr\prec\ka_1$.
This quantity does not depend on $j$, so if $e_1$, $e_2$ are edges
of $C$ meeting at a vertex of $C_1$, then $\la(e_1,C')=\la(e_2,C')$.
A similar argument, using $\ka_\rr\prec\ka_2$, gives the same
conclusion for edges meeting at a vertex of $C_2$.
Since $C$ is connected, it follows that $\la(e,C')$ is constant
on the edges of $C$. Let us write the common value as
$\kac(C,C')\mu_\cc(C')$.
Then, for each $i\in C_1$,
\[
 \sum_{k\in C_1'} \ka_1(i,k)\mu_1(k) = \kac(C,C')\mu_\cc(C'),
\]
and similarly, for each $j\in C_2$,
\[
 \sum_{\ell\in C_2'} \ka_2(j,\ell)\mu_2(\ell) = \kac(C,C')\mu_\cc(C').
\]
These last two identities establish $\ka_1\prec \kac$ and $\ka_2\prec\kac$
respectively, completing the proof.
\end{proof}

The statement of Theorem~\ref{updown} makes sense in the general case,
without the restriction to finite-type kernels, but the proof as
written does not. It is easy to adapt the proof that (ii) implies
(i) to the general case, but it does not seem to be easy to prove
that (i) implies (ii) in general. Indeed, it is not impossible that
this implication is false in the general case.

\medskip
Our main aim in this section is to prove the following result.

\begin{theorem}\label{bpdist}
Let $\ka_1$ and $\ka_2$ be finite-type kernels. Then $\pi_{\ka_1}=\pi_{\ka_2}$
if and only if there is a kernel $\kac$ with $\ka_1\prec\kac$ and $\ka_2\prec\kac$.
\end{theorem}

The proof will be a little involved (although most of the difficulties
are notational rather than actual), so we shall start by illustrating
a very simple special case of the basic idea.

Let $\ka$ be a kernel on the (finite) type space $(\sss,\mu)$.
In fact, our initial remarks will apply to general kernels.
Let $T\in \Tr_\infty$ denote a random rooted tree with distribution $\pi_\ka$,
so $T$ is obtained from $\bp_\ka$ by forgetting the types of the particles.
Also, let $T_t=T|_t$ denote the subtree of $T$ corresponding to the first
$t$ generations of $\bp_\ka$. We shall show that the distribution of $T_t$
is determined by a certain function of $\ka$, and vice versa.
We start by writing out the much simpler case $t=1$.

The tree $T_1$ is simply a star, so its distribution is determined by the
distribution of the degree of the root, i.e., the distribution
of the number $d_0$ of children of the initial particle of $\bp_\ka$.
As in~\cite{BJR}, for each $x\in \sss$, let us write
\[
 \la(x)=\la_1(x)= \int_\sss \ka(x,y)\dd\mu(y)
\]
for the {\em expected degree} of $x$, i.e., the expected number of children
of a particle of type $x$.
(The reason for the subscript $1$ will become clear later.)
Also, let $\La=\La_1$ denote the {\em (first order) expected degree distribution}
of $\ka$, i.e., the distribution of $\la(x)$ when $x$ is chosen
according to the distribution $\mu$.
From the definition of $\bp_\ka$, given the type of the root,
$d_0$ has a Poisson distribution with mean $\la(x)$.
Thus the unconditional distribution of $d_0$ is the mixed Poisson
distribution $\Po(\La)$, defined in the discrete case by
\begin{equation}\label{Pl}
 \Pr(\Po(\La)=k) = \sum_\la \Pr_{\La}(\la) \frac{\la^ke^{-\la}}{k!},
\end{equation}
where the sum is over the finite set of possible values of $\La$.
It follows that the distribution $\La$ determines the distribution
of $d_0$, and hence of $T_1$. The reverse is also true, since
the distribution of $\Po(\La)$ determines that of $\La$. (This is
trivial in the discrete case, using the tail of $\Pr(\Po(\La)=k)$
for large $k$ to identify the maximum possible value of
$\La$ and its probability, then subtracting off
the corresponding contribution to the sum in~\eqref{Pl},
identifying the next largest possible value, and so on.
The general case is not hard, using the generating function
of the distribution $\Po(\La)$.)

Let $x$ be a type with $\la(x)>0$.
From the definition of $\bp_\ka$, the types of the children
of a particle of type $x$ form a Poisson process on $\sss$
with intensity measure  $\mu_x$, defined
by $\dd\mu_x(y)=\ka(x,y)\dd\mu(y)$.
In order to understand the distribution of $T_2$, we consider
the {\em offspring expected degree distribution}
$\la_2(x)$, the image of $\mu_x(y)$ under the map $y\mapsto \la_1(y)$.
Thus, if $\mu_x$ were a probability measure, $\la_2(x)$
would be the distribution of $\la_1(Y)$ when $Y$
has the distribution $\mu_x$; in general, neither
$\mu_x$ nor $\la_2(x)$ is a probability measure:
they both have total mass $\mu_x(\sss)=\la_1(x)$.

Similarly, for $k\ge 3$, we define $\la_k(x)$ to be the image of the measure
$\mu_x(y)$ under the map $y\mapsto\la_{k-1}(y)$.
Thus
\[
 (\la_k(x))(A) = \int_{y\in \sss\::\: \la_{k-1}(y)\in A}  \dd\mu_x(y)
 = \int_{y\in \sss\::\: \la_{k-1}(y)\in A}  \ka(x,y)\dd\mu(y).
\]
Note that for a given $x$, $\la_1(x)=\la(x)$ is a real number,
$\la_2(x)$ is a measure on the reals, $\la_3(x)$ is a measure on the set of measures
on the reals, and so on. If $\ka$ is of finite type, then all these measures
are discrete.
By the {\em $k$-th order expected degree distribution} $\La_\ka$ of $\ka$, we mean
the distribution of $\la_k(x)$ when $x$ is chosen randomly with distribution $\mu$.

We shall deduce Theorem~\ref{bpdist} from the following lemma.

\begin{lemma}\label{l_claim1}
Fix $k\ge 1$, and
let $\ka$ be a finite-type kernel. Then
the distribution $\La_k$ determines the distribution of $T_k$ and vice versa.
\end{lemma}

The restriction to finite-type kernels is presumably not needed here, but simplifies
the proofs, avoiding any possible difficulties associated to choosing the right
notion of convergence. Note that we have already proved the case $k=1$.

Before proving Lemma~\ref{l_claim1}, let us show that Theorem~\ref{bpdist} does indeed follow.

\begin{proof}[Proof of Theorem~\ref{bpdist}]
If $\ka_1\prec\kac$ and $\ka_2\prec\kac$, then $\pi_{\ka_1}=\pi_{\kac}=\pi_{\ka_2}$;
the content of the theorem is the reverse implication.
We shall show that if $\ka$ is a finite-type kernel, then one can
define a (finite-type) kernel $\kac$ with $\ka\prec\kac$
in a way that depends only on $\pi_\ka$;
this clearly implies the result.

Given a finite-type kernel $\ka$ on $(\sss,\mu)$,
define an equivalence relation $\sim$
on $\sss$ by $x\sim y$ if $\la_k(x)=\la_k(y)$ for every $k$.
If $x\not\sim y$, then there is some smallest
$k=k(x,y)$ such that $\la_k(x)\ne \la_k(y)$.
Let $K$ be an upper bound on the set $\{k(x,y):x,y\in \sss, x\not\sim y\}$,
which exists since $\sss$ is finite.
Since $\la_{k+1}(x)$ determines $\la_k(x)$, we have $\la_K(x)\ne\la_K(y)$
whenever $x\not\sim y$,
so
\begin{equation}\label{iff}
 x\sim y \iff \la_K(x)=\la_K(y) \iff \la_{K+1}(x)=\la_{K+1}(y).
\end{equation}
Note that $K$ is determined by the set $\La_k$, $k=0,1,2,\ldots$:
we may take $K$ to be the smallest integer such that the distribution
of $\la_{K+1}$ (which then determines that of $\la_K$) has
property \eqref{iff}.

Let us define a type space $\sss_\cc$ whose elements are the possible values
of $\la_K(x)$, $x\in\sss$, and a corresponding $\mu_\cc$
given by the $\mu$-probability of the corresponding set of (equivalent)
types $x\in \sss$. Note that $\sss_\cc$ and $\mu_\cc$ depend
only on the distribution of $\la_K(x)$, and hence are determined
by $\La_{K+1}$. We map $\sss$ to $\sss_\cc$ in the obvious way;
this map is measure preserving by definition.
The distribution of $\la_K(y)$ over children $y$ of a particle
of type $x$ is determined by $\la_{K+1}(x)$, and hence,
from \eqref{iff}, by $\la_K(x)$. It follows that the distribution
of the $\sss_\cc$-types of the children depends only on the $\sss_\cc$-type
of $x$. This is exactly the statement that there is
a kernel $\ka_\cc$ on $(\sss_\cc,\mu_\cc)$ with $\ka\prec\ka_\cc$.
The kernel $\ka_\cc$ is determined by the distribution
of $\La_{K+1}(x)$, i.e., by $\La_{K+2}$. Hence, $\ka_\cc$
is determined by the sequence $\La_k$, $k=1,2,\ldots$.
Using Lemma~\ref{l_claim1} it follows that
$\ka_\cc$ is determined by the distribution
of $T_k$ for all $k$, and thus by $\pi_\ka$. As noted
above the existence of such a $\ka_\cc$ with $\ka\prec\ka_\cc$
suffices to prove the theorem.
\end{proof}

It remains to prove Lemma~\ref{l_claim1}. Note the lemma makes
two separate statements; in proving Theorem~\ref{bpdist} 
we only used one of these, that the distribution of $T_k$
determines that of $\La_k$. We shall prove Lemma~\ref{l_claim1} by induction;
for this we need both statements. In fact, to make the induction work,
we shall need to prove a little more.

Let $\ka$ be a kernel on the finite type-space $(\sss,\mu)$. The measure
$\mu$ plays two roles in the branching process $\bp_\ka$: it appears
in the distribution of the offspring of a particle, and also in the distribution
of the type of the initial particle. It will be convenient to generalize $\bp_\ka$ slightly
as follows: let $\mu_0$ be any probability measure on $\sss$, and let $\bp_\ka(\mu_0)$
be the branching process defined as $\bp_\ka$, but starting with a single particle
of type distributed according to $\mu_0$. Note that $\bp_\ka(\mu_0)$ depends
on $\mu$ as well as $\mu_0$, and that $\bp_\ka(\mu)=\bp_\ka$.

Let $\La_k(\mu_0)$ denote the distribution of $\la_k(x)$ when $x$ is chosen
randomly with distribution $\mu_0$, so $\La_k(\mu)=\La_k$.
Also, let $T_k(\mu_0)$
denote the random rooted tree obtained from the first $k$ generations
of $\bp_\ka(\mu_0)$ by forgetting the types of the particles.
The following lemma is slightly stronger than Lemma~\ref{l_claim1}, which
can be recovered by setting $\mu_0=\mu$.

\begin{lemma}\label{l_claim2}
Fix $k\ge 1$, let $\ka$ be a finite-type kernel on $(\sss,\mu)$,
and let $\mu_0$ be a probability
measure on $\sss$. Then
(i) the distribution $\La_k(\mu_0)$ determines the distribution
of $T_k(\mu_0)$, and
(ii) the distribution of $T_k(\mu_0)$ determines $\La_k(\mu_0)$.
\end{lemma}

\begin{proof}
We start with the easier statement, namely part (i), which we prove by induction on $k$;
as noted above, the case $k=1$ is trivial: $T_1$ is a star the degree
of whose root has the mixed Poisson distribution $\Po(\La_1(\mu_0))$.

Suppose then that $k\ge 2$ and that (i) holds with $k$ replaced by $k-1$.
It is easy to see that it suffices to prove (i) with $\mu_0$ concentrated on a
single type
$x$, in which case $\La_k(\mu_0)=\la_k(x)$.
Let us fix the type $x\in\sss$ of the root,
writing $\bp_\ka(x)=\bp_\ka(\delta_x)$ for the branching process $\bp_\ka$
started with a single particle of type $x$.

Let $X_1$ denote the first generation of $\bp_\ka(x)$.
Given $X_1$, the descendants of a particle $v$ in $X_1$ have the distribution
of $\bp_\ka(y)$, where $y$ is the type of $v$, and the subtrees corresponding
to different $v$ are independent. By induction, the distribution
of the first $k-1$ generations of the descendants of $v$ are determined
by $\la_{k-1}(y)$. Hence, given $X_1$, the conditional distribution
of $T_k$ depends only on the multiset $M=\{\la_{k-1}(y)\}$, where $y$
runs over the types in $X_1$.
Given the type $x$ of the root,
the types of the particles in $X_1$ form a Poisson process
on $\sss$ with intensity measure $\mu_x$.
Hence, $M$ is a Poisson process
on the appropriate space of distributions with intensity
measure $\la_k(x)$. In particular, the distribution of $M$,
and hence that of $T_k$, is determined by $\la_k(x)$, completing the proof of part (i)
by induction.

We now turn to the reverse implication, part (ii), which we again
prove by induction. For $k=1$ the argument is as before: the degree of the root
in $T_1(\mu_0)$ has a mixed Poisson distribution $\Po(\La_1(\mu_0))$,
which determines $\La_1(\mu_0)$.
Suppose then that $k\ge 2$, and that (ii) holds with $k$ replaced by $k-1$,
i.e., that for any probability measure $\mu_0'$ on $\sss$, the distribution
of $T_{k-1}(\mu_0')$ determines $\La_{k-1}(\mu_0')$.
This time we cannot simply condition on the type of the root,
as we are given the distribution of $T_k$ without types.
It is here that we shall use the fact that $\sss$ is finite.

For any realization $T$ of the random tree $T_k(\mu_0)$,
let $\D(T)$
denote the empirical distribution of the branches of $T$,
i.e., the subtrees
of $T$ of height $k-1$ rooted at the neighbours of the root
of $T$. Thus $\D(T)$ is a distribution on $\Tr_{k-1}$,
the set of rooted trees of height $k-1$, and the weight
$\D(T)$ assigns to a tree $T'$ is just the number of branches
of $T$ that are isomorphic to $T'$ divided by the total
number of branches.
Let $\tau_0$ denote the type of the root of $\bp_\ka(\mu_0)$,
and $d_0=|X_1|$ its degree.
Given that $\tau_0=x$ and $d_0=N$, the types
of the particles in $X_1$ are independent with the distribution
$\mut_x=\mu_x/\mu_x(\sss)=\mu_x/\la(x)$, the normalized version of $\mu_x$.
Hence their descendants are independent copies of the
branching process $\bp_\ka(\mut_x)$, and
the corresponding branches are $N$ independent samples
from the distribution $T_{k-1}(\mut_x)$.
Let us view the empirical distribution $\D(T_k(\mu_0))$
as a point $\vv$ in the space $[0,1]^{\Tr_{k-1}}$ taken
with the supremum norm. This point is random, since
it depends on the realization of $\bp_\ka(\mu_0)$.
From the law of large numbers and the comments above,
as $N\to\infty$ the random point $\vv$ obtained
after conditioning on $\tau_0=x$ and $d_0=N$ 
converges in probability to the single point $T_{k-1}(\mut_x)\in [0,1]^{\Tr_{k-1}}$.

Since all we are given is the distribution of $T_k(\mu_0)$, we cannot condition
on the type of the root. We can however condition on its degree, $d_0=|X_1|$.
Let $\lamax$ be the largest value of $\la(x)$ for $x\in\sss$ with
$\mu_0(x)>0$.
As $N\to\infty$,
\[
 \Pr(\tau_0=x\mid d_0=N) \to p_x = \mu_0(x)/\mu_0(\{y:\la(y)=\lamax\})
\]
if $\la(x)=\lamax$, and to zero otherwise.
From the comments above, as $N\to\infty$,
the distribution of $\vv=\D(T_k(\mu_0))$ given that $d_0=N$
tends to the discrete distribution taking each value
$T_{k-1}(\mut_x)$ with probability $p_x$. Hence,
the distribution $T_k(\mu_0)$ determines what
distributions are possible for $T_{k-1}(\mut_x)$ with $\la(x)=\lamax$,
and the probability of each (a sum of $p_{x'}$
over $x'$ such that $T_{k-1}(\mut_{x'})=T_{k-1}(\mut_x)$).

By induction, for each possible distribution $T_{k-1}(\mut_x)$
there is a unique corresponding distribution $\La_{k-1}(\mut_x)$.
Since we know $\lamax$, and the measure $\La_{k-1}(\mu_x)$ is simply obtained
by multiplying the probability measure $\La_{k-1}(\mut_x)$
by the constant factor $\lamax$, we recover the possible distributions
$\La_{k-1}(\mu_x)$ for $x$ with $\la(x)=\lamax$, and the $p$-probability
of each. In other words, we recover the distribution $\La_k(p)$
where $p$ is the probability measure defined by $p(\{x\})=p_x$.

By part (i),
the distribution $\La_k(p)$ determines the distribution
of $T_k(p)$, which is simply that of
$T_k(\mu_0)$ conditioned on $\tau_0$ lying in the set $\la(\tau_0)=\lamax$.
Since we recover this distribution {\em exactly}, we also
recover the conditional distribution of $T_k(\mu_0)$ given that 
$\la(\tau_0)<\lamax$, i.e., we recover the distribution
$T_k(\mu_0')$, where $\mu_0'$ is the distribution
$\mu_0$ conditioned on $\la(\cdot)<\lamax$.
Repeating the argument above, we can recover the part of $\La_k(\mu_0')$
coming from the largest possible $\la$-value of the root,
i.e., the part of $\La_k(\mu_0)$ coming form the second largest value, and so
on. This shows that the distribution of $T_k(\mu_0)$ does determine
that of $\La_k(\mu_0)$, completing the proof of (ii) by induction.
\end{proof}

Theorem~\ref{bpdist} shows that there are many examples of different
kernels that give rise to the same branching process, and hence to the
same distribution of tree counts in the corresponding random graphs
$G_{1/n}(n,\ka)$. One extremely special
case concerns homogeneous kernels: we say that $\ka$
is {\em homogeneous with degree $c$} if
$\int_y \ka(x,y)\dd y=c$ for almost every $x$.
In this case, $\bp_\ka$ seen without types becomes
a standard single-type Galton--Watson branching process $\bp_c$
in which each particle has a Poisson number of children with mean $c$.
Writing $c$ also for the constant kernel taking the value $c$,
Theorem~\ref{bpdist} shows that $\pi_\ka=\pi_c$
if and only if $\ka$ is homogeneous with degree $c$.
(This special case is essentially trivial, however: one need
consider only the first generation of the branching process.)

In the denser case, i.e., whenever
$np\to\infty$, Lemma~4.10 of~\cite{BRsparse} tells us that
the sequence $G_p(n,\ka)$ converges to $\ka$ in the
cut metric with probability~$1$, so it follows that the models
$G_p(n,\ka_1)$ and $G_p(n,\ka_2)$ are genuinely different unless
$\dc(\ka_1,\ka_2)=0$, i.e.,
unless $\ka_1$ and $\ka_2$ are equivalent
in the sense of \ssequiv, in which case the models
are trivially the same.
When $p=1/n$, we have seen that there are many pairs
$(\ka_1,\ka_2)$ of kernels where the random graphs
$G_{1/n}(n,\ka_1)$ and $G_{1/n}(n,\ka_2)$ are
presumably different, but are not distinguished by their
tree counts. This suggests that we need better methods of
distinguishing very sparse graphs.
However, as we shall see in this next sections, while this is true, the
question of which pairs of kernels give rise to equivalent models is not so
easy.

\section{The partition metric}\label{ss_par}
In the spirit of the rest of the paper, we say that two graphs with $n$
vertices are {\em essentially the same} if one can be changed
into a graph isomorphic to the other by adding and deleting
$o(pn^2)$ edges, where $p=p(n)$ is our normalizing function,
as usual. (Of course, the definition makes formal sense
only for two sequences.) Otherwise, they are {\em essentially different}.
In all previous sections, graphs that were essentially the same were
treated as equivalent, in the sense that their distance in any of the
metrics we considered tends to zero.

Let $p=1/n$, and let $\ka$ be a kernel whose corresponding branching
process always dies out. In the notation of Bollob\'as, Janson
and Riordan~\cite{BJR}, we assume that the operator $T_\ka$
corresponding to the kernel $\ka$ satisfies $\norm{T_\ka}\le 1$,
i.e., $\ka$ is {\em (weakly) subcritical}.
From the results in~\cite{BJR}, almost all vertices of $G_{1/n}(n,\ka)$
are in small tree components: more precisely, given any $\eps>0$,
there is a $K$ such that, whp, all but at most $\eps n$ vertices
of $G_{1/n}(n,\ka)$ are in tree components with size at most $K$.
Furthermore, the asymptotic number of copies of a given tree
$T$ in $G_{1/n}(n,\ka)$ is determined by the probability of $T$
in the distribution $\pi_\ka$.
It follows that if $\ka_1$ and $\ka_2$ are subcritical
kernels, then $G_{1/n}(n,\ka_1)$ and $G_{1/n}(n,\ka_2)$ are (whp)
essentially the same if and only if $\pi_{\ka_1}=\pi_{\ka_2}$.
Hence, in the subcritical case, the random graph $G_{1/n}(n,\ka)$
depends only on the branching process $\bp_\ka$. Of course,
this rather trivial observation does not extend to the supercritical case.

Given two real numbers $a,b\ge 0$, let $\ka_{a,b}$ denote the $2$-by-$2$
`chessboard' kernel defined as follows:
\begin{equation}\label{cbk}
 \ka_{a,b}(x,y) = \left\{
\begin{array}{ll}
  a & \hbox{if } x<1/2,\  y<1/2 \hbox{ or } x\ge 1/2,\  y\ge 1/2,\\
  b & \hbox{otherwise.}
\end{array}\right .
\end{equation}
To form the random graph $G_p(n,\ka_{a,b})$, we partition the vertex set randomly
into two parts, and then take each cross-edge to be present with probability $bp$,
and each other edge with probability $ap$. 
Note that $\ka_{a,b}$ is homogeneous with constant $(a+b)/2$.
Also, if $a=b$, then $\ka_{a,b}$ is simply the constant kernel
taking the value $a=b$.

For $p=1/n$, perhaps the simplest example of two sequences of essentially
different graphs not distinguished by their
tree counts is given by the random graphs $G_{1/n}(n,\ka_{2,2})$ and $G_{1/n}(n,\ka_{4,0})$,
i.e., the usual Erd\H os--R\'enyi random graph $G(n,2/n)$ and (essentially)
the random bipartite graph $G(n/2,n/2;4/n)$.
How do we know that these graphs are different? For the obvious reason that
one is bipartite, with almost equal vertex classes,
while the other is not. Indeed,
the smallest balanced cut in $G(n,2/n)$ has size
of order $\Theta(n)$: this follows, for example,
from the result of Luczak and McDiarmid~\cite{LMcD}
that removing $o(n)$ edges from the giant component of $G(n,c/n)$, $c>1$,
leaves a connected component with only $o(n)$ fewer vertices than
the original giant. Note that one has to be a little careful here:
writing $\rho(c)$ for the largest solution to $\rho=1-e^{-c\rho}$,
so $\rho(c)n$ is the typical size of the giant component in $G(n,c/n)$,
we need $\rho(c)>1/2$; otherwise, it is easy to construct
a balanced cut with $o(n)$ edges across it.
Note that both $G(n,2/n)$ and $G(n/2,n/2;4/n)$ have balanced
cuts with a range of sizes: the difference between the two graphs
can be seen in the difference between these ranges.

In the discussion above we considered balanced partitions
of the vertex set of a graph into two pieces. Of course, 
it makes sense to consider any fixed number $k$ of pieces; this
leads us to consider the {\em partition metric} $\dP$
defined in~\cite{BRsparse} for any normalizing
function $p=p(n)$, but in fact motivated by the present case.
Let us recall the definitions from~\cite{BRsparse},
writing $d_p(U,W)$ for the normalized density
of edges from $U$ to $W$ in $G_n$ as in~\eqref{dpdef}.

Fix throughout a normalizing function $p=p(n)$ and a constant $C>0$;
we shall only consider graphs $G_n$ with $n$ vertices and at most
$Cpn^2/2$ edges.

Given a graph $G_n$ with $|G|=n$ and $e(G)\le Cpn^2/2$,
for each partition $\Pi=(P_1,\ldots,P_k)$ of $V(G_n)$
into non-empty parts let
$M_\Pi(G_n)=(d_p(P_i,P_j))_{1\le i,j\le k}$ be the matrix encoding
the normalized densities of edges between the parts of $\Pi$.
Since $M_\Pi(G_n)$ is symmetric, we may think of this
matrix as an element of $\R^{k(k+1)/2}$.
For $2\le k\le n$, let
\[
 \M_k(G_n) = \{ M_\Pi(G_n) \} \subset \R^{k(k+1)/2},
\]
where $\Pi$ runs over all balanced partitions of $V(G_n)$ into
$k$ parts, i.e., all partitions $(P_1,\ldots,P_k)$ with $|P_i|-|P_j|\le 1$.

Recall that $e(G_n)\le Cpn^2/2$, so $e(U,W)\le e(V(G_n),V(G_n))=2e(G_n) \le Cpn^2$.
Since each part of a balanced
partition has size at least $n/(2k)$, the entries of
any $M_\Pi(G_n)\in \M_k(G_n)$ are thus bounded by $C_k=(2k)^2C$,
and $\M_k(G_n)$ is a subset of the compact space $\Mk=[0,C_k]^{k(k+1)/2}$.

As in~\cite{BRsparse},
let $\C_0(\Mk)$ denote the set of non-empty compact subsets of $\Mk$,
and let $\dH$ be the Hausdorff metric on $\C_0(\Mk)$,
defined with respect to the $\ell_\infty$ distance, say.
Thus
\[
 \dH(X,Y)=\inf \{ \eps>0: X^{(\eps)}\supset Y,\ Y^{(\eps)}\supset X\},
\]
where $X^{(\eps)}$ denotes the $\eps$-neighbourhood of $X$ in the $\ell_\infty$ metric.
Since $(\Mk,\ell_\infty)$ is compact, by standard results (see, for example, Dugundji~\cite[p.\ 253]{Dugundji}),
%it's exercise XI.4.6
the space $(\C_0(\Mk),\dH)$ is compact.
For technical reasons we add the empty set to $\C_0(\Mk)$, considering
$\C(\Mk)=\C_0(\Mk)\cup\{\emptyset\}$,
setting $\dH(\emptyset,X)=C_k$, say, for any $X\in \C(\Mk)$, so
the empty set is an isolated point in the compact metric space $(\C(\Mk),\dH)$.

Finally, let $\C=\prod_{k\ge 2} \C(\Mk)$, and let 
$\M:\F\mapsto \C$ be the map defined by
\[
 \M(G_n) = (\M_k(G_n))_{k=2}^\infty
\]
for every graph $G_n$ with $n$ vertices and at most $Cpn^2/2$ edges,
noting that $\M_k(G_n)$ is empty if $k > n$.
Then we may define the {\em partition metric} $\dP$ by
\[
 \dP(G,G') = d(\M(G),\M(G')),
\]
where $d$ is any metric
on $\C$ giving rise to the product topology.

As noted in~\cite{BRsparse},
$\dP$ is indeed a metric on the set $\F$ of isomorphism classes of finite graphs.
Also, since each space $(\C(\Mk),\dH)$ is compact,
$(G_n)$ is Cauchy with respect to $\dP$ if and only if
there are non-empty compact sets $Y_k\subset \Mk$ such that
$\dH(\M_k(G_n),Y_k)\to 0$ for each $k$.
In particular, convergence in $\dP$ is equivalent to convergence
of the set of partition matrices for each fixed $k$.

In the dense case, a metric equivalent to $\dP$ was introduced
independently by Borgs, Chayes, Lov\'asz, S\'os and Vesztergombi~\cite{BCLSV:2}.

The definitions above may appear rather unnatural: the set $\M_k(G_n)$
of possible density matrices is perhaps more naturally seen as a multiset,
with one element for each of the $N_{n,k}$ balanced partitions
of $[n]$ into $k$ (ordered) parts; the Hausdorff metric ignores
the multiplicities of the points of these sets.
For multisets $S$, $S'$ in a metric space $(X,d)$ with $|S|=|S'|=N$,
(a version of) their {\em matching distance} is given by
\begin{equation}\label{dMdef}
 \dM(S,S') = \min_\phi \max_{x\in S} d(x,\phi(x)),
\end{equation}
where $\phi$ runs over all {\em bijections} between $S$ and $S'$ (as multisets).
In other words, we pair up the points of $S$ with those of $S'$
and measure the maximum distance between corresponding points, minimized
over pairings.
For graphs $G_n$, $G_n'$ with the same number of vertices, it may make
sense to use $\dM$ instead of $\dH$ to measure the distance between
$\M_k(G_n)$ and $\M_k(G_n')$, defining the {\em partition matching
distance} $\dPM(G_n,G_n')$ accordingly.

In fact, $\dM$ can easily be extended to multisets $S$, $S'$ with different
numbers of elements: simply replace each point of $S$ by $|S'|$
points, and each point of $S'$ by $|S|$ points, then compute
the matching distance. In other words, minimize over `fractional bijections'
$\phi$ above; if $|S|=|S'|$ this does not change the minimum distance.
Extending $\dM$ in this way, we could thus define $\dPM(G,G')$
for any two graphs $G$, $G'$. However, the resulting
metric behaves much less well than $\dP$: for example, not all
(sparse) sequences will have subsequences that are Cauchy, in contrast
to $\dP$, which, as noted above, is defined via a metric on a compact space.
Even worse, it is easy to see that the sequence $(G(n,c/n))$ will not converge
in the partition matching metric.

The matching distance and the Hausdorff distance share what might
appear to be an undesired property: they are strongly influenced by atypical
partitions $\Pi$. Surely, for multisets, it would be more natural
to weight points by their multiplicity, replacing \eqref{dMdef} by
\[
 \dM'(S,S')=\min_\phi \frac{1}{|S|} \sum_{x\in S} d(x,\phi(x)).
\]
Recall, however, that our main aim in introducing the partition
metric is to distinguish in a sensible way such pairs of graphs as the
uniform random graph $G_n=G(n,2/n)$ and the random bipartite graph
$G_n'=G(n/2,n/2;4/n)$. It is very easy to see that for almost all partitions
of $V(G_n')$ into a fixed number $k$ of parts, each part contains
almost equal numbers of vertices from the two vertex classes of $G_n'$.
It follows that almost all (in the multiset sense) density matrices
arising from $G_n'$ are very close to each other and to almost
all density matrices arising from $G_n$. The whole point of the partition
metric is that it accords significant weight to atypical partitions,
in particular, to the partition of $G_n'$ corresponding to its two
vertex classes. For this reason we now return to considering
$\M_k(G_n)$ as a set rather than a multiset, and stay
with the definition of $\dP$ above based on the Hausdorff metric.

As noted in~\cite{BRsparse},
we may extend the map $\M$, and hence $\dP$, to kernels in a natural way: instead
of partitioning the vertex set into $k$ almost equal parts, we partition
$[0,1]$ into $k$ exactly equal parts. We omit the details, noting
only that as shown by Borgs, Chayes, Lov\'asz, S\'os and Vesztergombi~\cite{BCLSV:2},
one should take the closure of the resulting set
of density matrices, which need not itself be closed; see their Example 4.4.

As for the cut metric, it is easy to check that it makes little difference
whether we define $\dP$ for graphs directly, or by going via kernels;
the next lemma is from \cite[Section 6]{BRsparse}.

\begin{lemma}\label{dPsame}
Let $p=p(n)$ satisfy $p\ge 1/n$, and let $(G_n)$ be a sequence
of graphs with $e(G_n)=O(pn^2)$ and $\Delta(G_n)=o(pn^2)$.
Then $\dP(G_n,\ka_{G_n})\to 0$ as $n\to\infty$.\noproof
\end{lemma}

Although the partition metric was defined in~\cite{BRsparse},
the real motivation for its study comes from the present
extremely sparse setting. Indeed, as shown in~\cite[Section 6]{BRsparse},
whenever $np\to\infty$ then (for sequences $G_n$
satisfying a mild additional condition) $\dc$ and $\dP$ are
essentially equivalent.

\begin{theorem}\label{dpdc}
Let $np\to\infty$, and let $(G_n)$ be a sequence of graphs with $|G_n|=n$
satisfying the bounded density assumption~\cite[Assumption 4.1]{BRsparse}.
Let $\ka$ be a bounded kernel. Then $\dP(G_n,\ka)\to 0$ if and only
if $\dc(G_n,\ka)\to 0$.\noproof
\end{theorem}

\subsection{The partition metric and random graphs}

Let us return to our main focus in this paper, the extremely
sparse case $p=1/n$. Our hope was that in this setting the partition
metric might play the role of the cut metric in the denser setting,
showing, for example, that a random sequence $(G_{1/n}(n,\ka))$ has a
limit with probability $1$, and that this limit is different for different
$\ka$.

For $\dc$, in the denser setting, the limit was $\ka$ itself,
but here we cannot hope for this. Indeed, it is very easy to check
that in $G(n,c/n)$, the largest and smallest balanced cuts
have sizes that differ from $cn/4$ (the expected size of a random cut)
by order $n$. Indeed, using the greedy algorithm to assign
vertices one by one to a part of the bipartition, it is easy
to construct such a cut. For the best current bounds
on the largest cut in $G(n,c/n)$, see \cite{BCP},
\cite{CGHS1} and \cite{CGHS2} (for related results, see \cite{BBi},
\cite{KM}, \cite{LMcD}, \cite{FKL}, \cite{DSW}).
As pointed out to us by Christian Borgs,
one can describe the problem in the language of statistical physics:
when $p=1/n$, the entropy and energy terms balance. More precisely, a given cut is
`only' exponentially unlikely to have $1\%$ fewer
edges than the expected number (as opposed to superexponentially unlikely
when $np\to\infty$), but there are exponentially many cuts, so
some atypical cuts will exist.

Despite the observation above, it is extremely likely that
$(G_{1/n}(n,\ka))$ is Cauchy with respect to $\dP$ with probability $1$,
i.e., that there is a limit point (depending only on $\ka$)
in the space $\prod_k \C(\Mk)$ defined above, even though
this limit is not given by $\ka$ in the way one might expect.

\begin{conj}\label{q3}
For any bounded kernel $\ka$, the random
sequence $G_{1/n}(n,\ka)$ is Cauchy with respect to $\dP$
with probability $1$.
\end{conj}

Note that the analogue of Conjecture~\ref{q3} with $np\to\infty$ is
trivial, since Chernoff's inequality shows that $G_p(n,\ka)$
and $\ka$ are close in the cut metric, and hence in $\dP$.
While Conjecture~\ref{q3} is very likely to be true, it may also be rather hard
to prove, since it would imply the convergence of many quantities
associated to $G(n,c/n)$ (such as the normalized size of the largest
cut) that are not known to converge. For example,
one can use partitions to pick out the size of the largest
independent set within $o(n)$ (in this sparse setting, an almost independent
set, spanning $o(n)$ edges, contains an independent set of almost
the same size), so Conjecture~\ref{q3} implies the following statement:
there is a function $\beta(c)$ such that the independence
number of $G(n,c/n)$ is $\beta(c)n+\op(n)$ as $n\to\infty$ with $c$
fixed; in particular, the independence number is concentrated for
each $n$. While concentration is well known and easy to prove,
at the moment it is not known that the
independence number can't `jump around' as $n$ increases, although
of course this is extremely implausible. For a discussion of this,
see, for example, Gamarnik, Nowicki and Swirszcz~\cite{GNS}; surprisingly,
for $c\le e$, convergence is known: it follows from results
of Karp and Sipser~\cite{KarpSipser}.

The same concentration applies to $\dP$, as shown by the following result.
\begin{theorem}\label{dPconc}
Let $\ka$ be a bounded kernel, let $k\ge 2$ be fixed, and let $G_n=G_{1/n}(n,\ka)$.
For each $n$ there is a set $Y_n\subset \Mk$ such that
$\dH(\M_k(G_n),Y_n)$ converges to $0$ in probability.
\end{theorem}
\begin{proof}
Since $\Mk$ is compact, from the definition of the Hausdorff
metric it is enough to show that for any given point $M\in \Mk$ 
the random variable
\[
 Z_n=d(M,\M_k(G_n)) = \inf_{M'\in \M_k(G_n)} d(M,M')
\]
is concentrated
around its mean as $n\to\infty$. For each $\eps>0$, taking an $\eps$-net in $\Mk$,
one can then find (discrete) sets $Y_{n,\eps}$ such that
\begin{equation}\label{necl}
 \dH(\M_k(G_n),Y_{n,\eps})\le 3\eps
\end{equation}
holds whp. Since \eqref{necl} holds whp for any fixed $\eps$, it also
holds whp for some function $\eps(n)$ tending to zero; taking
$Y_n=Y_{n,\eps(n)}$ then gives the result.

Roughly speaking,
since the real-valued random variable $Z_n$ changes by order $1/n$
if we add or delete an edge of $G_n$, concentration of $Z_n$
follows by standard martingale arguments. One must be a little careful,
however, for two reasons. Firstly, we cannot afford to use the edge-exposure
martingale, since it has too many steps. Using vertex exposure, one
must consider the possibility of large degrees. Secondly, the `type variables'
$x_1,\ldots,x_n$ introduce some dependence between edges.
There are many ways of working around these problems. One possibility
is as follows.

Let $c=\sup \ka<\infty$. We may couple $G_n$ and $G_n'=G(n,c/n)$
in a natural way so that $G_n\subset G_n'$. Indeed, first construct
$G_n'$, then choose the types $x_1,\ldots,x_n$, then keep
each edge $ij$ of $G_n'$ with probability $\ka(x_i,x_j)/c$, independently
of the others. It is easy to see that the set of edges remaining
has the distribution of $G_n$. (This construction is also used
by Bollob\'as, Janson and Riordan~\cite{BJRclust}.)

For the moment, let us condition on $G_n'$, fixing a possible
graph $G_n'$ with $e(G_n')\le 10 cn$ and $\Delta(G_n')\le n^{1/10}$, say.
We can construct $G_n$ from $G_n'$ using a sequence of $n+2e(G_n')$
independent uniform $U[0,1]$ random variables: one, $x_i$, for each vertex,
and one, $w_{ij}$, for each edge $ij$ of $G_n'$. Indeed, we simply keep the edge
$ij$ of $G_n'$ if and only if $w_{ij}\le \ka(x_i,x_j)/c$.
Changing one of the $w_{ij}$ adds or deletes at most one edge of $G_n$,
and so changes $Z_n$ by at most $k^2/n$. Changing one of the $x_i$ only affects
the presence in $G_n$ of edges of $G_n'$ incident with vertex $i$.
By assumption, there are at most $n^{1/10}$ such edges, so changing
$x_i$ changes $Z_n$ by at most $k^2n^{-9/10}$. Since we make only $O(n)$ choices,
it follows using the Hoeffding--Azuma inequality that, given $G_n'$
with the above properties, $Z_n$ is concentrated about its
mean $\E(Z_n\mid G_n')$.

Since almost every $G_n'=G(n,c/n)$ satisfies the bounds on $e(G_n')$ and
$\Delta(G_n')$ above, it remains only to show that $\E(Z_n\mid G_n')$ is concentrated.
For this it is more convenient to consider the graph $G_n''$ obtained as follows:
choose $cn/2$ edges independently and uniformly at random from all $\binom{n}{2}$
possible edges, deleting any repeated edges. It is easy to see that $G_n''$ and
$G(n,c/n)$ may be coupled to agree within $\op(n)$ edges. Since
adding or deleting an edge of $G$ changes $\E(Z_n\mid G_n'=G)$ by at most $k^2/n$,
it suffices to prove concentration of $\E(Z_n\mid G_n')$ when $G_n'$
has the distribution $G_n''$. But this is immediate from the Hoeffding--Azuma
inequality, since $G_n''$ is constructed from $O(n)$ independent choices
each of which changes this expectation by at most $k^2/n$.
\end{proof}

The result above shows that the random sets $\M_k(G_n)$ become concentrated 
as $n\to\infty$. The problem is that the points they become concentrated
around might in principle jump around as $n$ varies.

Even without a proof of Conjecture~\ref{q3}, it still makes good
sense to ask whether $\dP$ at least separates different
random graph models $G_{1/n}(n,\ka)$.
As before, let us write $\dedit(G_1,G_2)$ for the normalized edit distance between
two graphs $G_1$, $G_2$ with $|G_1|=|G_2|=n$, i.e., for $1/(pn^2)$ times the minimal 
number of edge changes (additions or deletions) that must be made
to $G_1$ to produce a graph isomorphic to $G_2$.
By a {\em random graph model} we simply mean a sequence of probability measures on
the sets of $n$-vertex graphs.
We say that two random graph models are {\em essentially equivalent}
if one can couple the corresponding $n$-vertex random graphs $G_n$ and $G_n'$
so that $\E(\dedit(G_n,G_n'))=o(1)$.

\begin{conj}\label{q4}
Let $p=1/n$, and let $\ka_1$ and $\ka_2$ be bounded kernels such that the models
$G_{1/n}(n,\ka_1)$ and $G_{1/n}(n,\ka_2)$ are not essentially equivalent.
Then the expected $\dP$ distance between $G_{1/n}(n,\ka_1)$ and $G_{1/n}(n,\ka_2)$
is bounded below as $n\to\infty$.
\end{conj}
Note that the distance between the $n$ vertex graphs is concentrated
by Theorem~\ref{dPconc}.
Of course, any proof of Conjecture~\ref{q4} is likely to involve understanding
for which pairs of kernels the corresponding models $G_{1/n}(n,\ka)$
are essentially equivalent. We discuss this briefly in the next section.

\section{Which kernels give the same random graphs?}

We have already seen a rather simple example of two kernels that are not
equivalent (in the sense of \ssequiv), which nonetheless
give rise to essentially equivalent sparse random graphs: we may take
any two non-equivalent kernels $\ka_1$, $\ka_2$ corresponding to the same
subcritical branching process. Of course, the corresponding random
graphs have a rather simple structure, since they are made up of (essentially)
only small tree components.
Unfortunately, (or interestingly, depending on ones point of view)
a simple modification of this example gives examples
with more complex structure.

\begin{example}\label{scc}
Let $c>1$ be constant, and let $\ka_1$ and $\ka_2$ be bounded kernels
corresponding to the same subcritical branching process. Writing $c+\ka_i$ for the 
pointwise sum of the constant kernel $c$ and the kernel $\ka_i$,
the models $G_{1/n}(n,c+\ka_i)$, $i=1,2$, are essentially equivalent.
To see this, we realize $G_{1/n}(n,c+\ka_i)$ by first constructing
$G_{1/n}(n,\ka_i)$, and then adding each non-edge with probability $c/n$,
and then adding each non-edge with a tiny probability to get the edge
probabilities exactly right. We can ignore the last step since
it adds $\Op(1)=\op(n)$ edges. Constructing the graphs $G_{1/n}(n,\ka_i)$
first, we may relabel the vertices of these graphs so that they coincide
apart from $\op(n)$ edges. We may then add each possible edge
to both graphs simultaneously with probability $c/n$ to obtain (essentially)
the required coupling of the graphs $G_{1/n}(n,c+\ka_i)$.
\end{example}

In general, we believe the following is an interesting question.

\begin{question}\label{kkq}
For which pairs of supercritical kernels $\ka_1$, $\ka_2$
are the models $G_{1/n}(n,\ka_1)$ and $G_{1/n}(n,\ka_2)$
essentially equivalent?
\end{question}

Certainly, any such pair must satisfy $\pi_{\ka_1}=\pi_{\ka_2}$,
otherwise the models are distinguished by their tree counts.
A simple answer to Question~\ref{kkq} would be important
for the general understanding of the sparse inhomogeneous
model of Bollob\'as, Janson and Riordan~\cite{BJR}.

Since Question~\ref{kkq} is rather open ended, let us focus on one
particular example: the pair consisting of the constant
kernel $c$ and the kernel $\ka_{c+\delta,c-\delta}$ defined
in~\eqref{cbk}, with $0<|\delta|<c$. The cases $\delta$ positive
and $\delta$ negative may behave differently, although we do not expect this
to be the case.
For $0<\delta'<\delta$, or $0>\delta'>\delta$, one can construct
$G_{1/n}(n,\ka_{c+\delta',c-\delta'})$ from $G_{1/n}(n,\ka_{c+\delta,c-\delta})$
by deleting each edge independently with a certain probability, and then
adding in each non-edge with an appropriate probability.
It follows that if $G_{1/n}(n,\ka_{c+\delta,c-\delta})$ and $G(n,c/n)$
are essentially equivalent, then so are $G_{1/n}(n,\ka_{c+\delta',c-\delta'})$ and $G(n,c/n)$.
Hence there is an interval $I(c)$ such that
$G_{1/n}(n,\ka_{c+\delta,c-\delta})$ and $G(n,c/n)$ are essentially equivalent for all
$\delta\in I(c)$, but for no $\delta\in[-c,c]\setminus I(c)$.

The construction in Example~\ref{scc} shows that $[-1,1]\subset I(c)$ whenever
$c>1$. On the other hand, it is not hard to show that for $c$ large,
the most extreme balanced cuts in $G(n,c/n)$ contain $cn/4\pm\Theta(\sqrt{c}n)$
edges; for the best known bound on the largest cut see \cite{BCP}, \cite{CGHS1} and \cite{CGHS2}.
Since a typical $G_{1/n}(n,\ka_{c+\delta,c-\delta})$ clearly contains a balanced
cut with $(c-\delta)n/4+\op(n)$ edges, it follows that $I(c)\subset[-A\sqrt{c},A\sqrt{c}]$
for $c$ large, where $A$ is an absolute constant.
It is not hard to convince oneself that the second bound is closer
to the truth, although we do not have a proof. In fact, we believe
that the endpoints of $I(c)$ are exactly $\pm\sqrt{c}$ for every
$c\ge 1$, although we make no guess as to whether these endpoints
are included.

\begin{conj}\label{c5}
Let $c>1$ and $-c \le \delta\le c$ be constants.
If $\delta<\sqrt{c}$, then the models
$G_{1/n}(n,\ka_{c+\delta,c-\delta})$ and $G(n,c/n)$ are essentially equivalent.
If $\delta>\sqrt{c}$, then they are not.
\end{conj}

The model $G_{1/n}(n,\ka_{c+\delta,c-\delta})$ is a special case of the {\em planted
bisection} model $G(n;p,p')$: for any $p=p(n)$ and $p'=p'(n)$, the graph
$G(n;p,p')$ is constructed by partitioning its vertex set $[n]$ at random
into two (almost) equal parts, and then joining any two vertices in the same
part with probability $p$, and two vertices in different parts with probability $p'$.
The question of reconstructing the vertex partition given only the graph
has received considerable attention, generally with emphasis on polynomial-time
algorithms for $p$, $p'$ satisfying suitable conditions; 
see, for example, Boppana~\cite{Boppana}, and, for a linear expected
time algorithm, Bollob\'as and Scott~\cite{BS}. Most such results are
for graphs with average degree tending to infinity, but Coja-Oghlan~\cite{CO}
proved results that include the extremely sparse case, showing that one can find
a minimum balanced cut in $G_{1/n}(n,\ka_{c+\delta,c-\delta})$ in polynomial 
time whenever $\delta>\Theta(1+c\log(c))$. The connection with Conjecture~\ref{c5}
is rather loose, but nonetheless interesting.

Let us present another question that does seem to be closely related to
Conjecture~\ref{c5}.

\begin{question}\label{qroot}
When does the branching process $\bp_{\ka_{c+\delta,c-\delta}}$ forget the type of the
root?
\end{question}

To spell out what we mean by Question~\ref{qroot},
let $T_t$ be the random labelled tree corresponding
to the first $t$ generations of $\bp=\bp_{\ka_{c+\delta,c-\delta}}$, with each
vertex labelled by the type of the corresponding particle.
Here we take the types to be $1$ and $2$, corresponding to $x\le 1/2$
and $x>1/2$. Note that each particle has $\Po((c+\delta)/2)$ children
of its own type, and $\Po((c-\delta)/2)$ of the opposite type.
Recall that, since $\pi_{\ka_{c+\delta,c-\delta}}=\pi_{c}$, the distribution of $T_t$ without labels
is simply that of the first $t$ generations of a single-type branching process.
Let $U_t$ be the tree obtained from $T_t$ by forgetting the labels
of all vertices {\em other than those at distance $t$ from the root}, and let $p_t$ be the probability,
conditional on $U_t$, that the root has type $1$, so $p_t$ is a random
variable depending on $\bp$ via $U_t$.
We say that $\bp$ {\em forgets the type of the root} if $p_t\pto 1/2$
as $t\to \infty$. It is easy to couple
the branching processes $\bp(1)$ and $\bp(2)$
started with particles of type $1$ and $2$, respectively, so that
the tree structures always agree, and the expected number of label mismatches
in level $t$ is $\delta^t$. It follows that $\bp$ forgets the type of the root
if $\delta\le 1$. Although we have not checked the details,
we believe that $\bp$ forgets the type of the root
if and only if $\delta\le \sqrt{c}$; this is based
on linearizing the natural recurrence describing the distribution
of $p_{t+1}$ in terms of that of $p_t$. (Actually, one works
with the distribution of $p_t$ given that the root has type $1$.)
This argument certainly
shows that $\bp$ cannot forget the type of the root if $\delta>\sqrt{c}$;
the reverse implication is not so clear.

Although we certainly have no proof, it seems likely that if
$\bp=\bp_{\ka_{c+\delta,c-\delta}}$ forgets the type of the root,
then the models $G_{1/n}(n,\ka_{c+\delta,c-\delta})$ and $G(n,c/n)$ are essentially equivalent.
Roughly speaking, suppose that, given the global structure of
$G_{1/n}(n,\ka_{c+\delta,c-\delta})$, seen without types, we can somehow
form a good guess as to which vertices at graph distance $100$ from a given
vertex $v$ are of type $1$ and which of type $2$. Even then, $v$ itself
is (almost) equally likely to be of either type. This strongly suggests
that one can get essentially no information about the vertex types from 
the graph, and hence that the types do not matter to the graph. This 
vague heuristic is very far from a proof, however!

In summary, it seems very likely that the answers to Conjecture~\ref{c5}
and Question~\ref{qroot} are closely related. In turn they 
may well be related to the question of when the maximum/minimum balanced
cut distinguishes $G_{1/n}(n,\ka_{c+\delta,c-\delta})$ from $G(n,c/n)$.
We do not even have a guess as to the form of a more general answer
to Question~\ref{kkq}.

\section{General extremely sparse graphs}\label{ss_nontree}

So far, we have mainly considered locally acyclic
graphs (the exception is Section~\ref{ss_par}).
This is natural when considering the metric $\de$,
for which only counts of trees make sense when using
the $p=1/n$ special case of
the normalization used in~\cite{BRsparse}.
It is also natural when considering
the random graph model $G_{1/n}(n,\ka)$.
However, there are many natural sequences of graphs with $\Theta(n)$
edges that are not treelike. A simple example
is given by taking $G_n$ to consist of $n/3$
vertex disjoint triangles, for $n$ a multiple of three.
Clearly, this graph is not close to any locally treelike graph.
It would be nice to have a notion of when two general graphs
with $\Theta(n)$ edges are close, as well as a more
general random model generating such graphs.

What distinguishes the union of $n/3$ triangles from a Hamilton
cycle $C_n$, say? The simplest answer is the number of triangles.
Throughout this section we consider sequences $(G_n)$
with exponentially bounded tree counts, i.e., we assume
that there is a constant $C$ such that
$\limsup_{n\to\infty} t(T,G_n)\le C^{e(T)}$ for every tree
$T$. This condition is certainly satisfied if the graphs
$G_n$ have bounded maximum degree, for example, and the
reader may wish to think of this case for simplicity. 
In fact, as in Section~\ref{treecounts}, something
weaker than exponential boundedness probably suffices,
but exponential boundedness is a natural assumption.
If $(G_n)$ has exponentially bounded (or indeed simply bounded)
tree counts, then the number of embeddings or homomorphisms
from any fixed graph $F$ into $G_n$ is $O(n)$.

For each fixed $F$, let
\[
 \ts(F,G_n) = \emb(F,G_n)/n,
\]
and let $\ttt(F,G_n)=\hom(F,G_n)/n$. As in Section~\ref{treecounts}
or in~\cite{BRsparse},
we can use the (now differently normalized) subgraph
counts $\ts$ or $\ttt$ to define a metric $\detl$, by first mapping
$G_n$ to
\[
 \ts(G_n) = (\ts(F,G_n))_{F\in \F} \in [0,\infty)^\F.
\]
As noted above, each coordinate is bounded for the sequences
we consider, so $\ts$ maps into a compact subset $X$ of $[0,\infty)^\F$.
Using any metric $d$ on $X$ giving rise
to the product topology, we may then define $\detl(G,G')=d(\ts(G),\ts(G'))$.

As in Section~\ref{treecounts}, one can view the (limiting) counts
$\ts(F,G_n)$ as giving the moments of a certain distribution that
we could instead study directly.
Let $\Gr_t$ be the set of isomorphism classes of connected finite
rooted graphs with radius at most $t$,
i.e., with every vertex at graph distance at most $t$ from the root.
Also, let $\Gr$ be the set of all locally finite rooted (finite or infinite) graphs.
As in Section~\ref{treecounts}, for $F\in \Gr_t$ let
\[
 p(F,G_n) = p_t(F,G_n) = \Pr\bb{ \Ga_{\le t}(v) \isom F},
\]
where $v$ is a vertex of $G_n$ chosen uniformly at random,
and $\Ga_{\le t}(v)$ is the subgraph of $G_n$ induced by the vertices
within distance $t$ of $v$, viewed as a rooted graph with root $v$.
Since each $p_t(F,G_n)$ lies in $[0,1]$, trivially, any sequence $(G_n)$
has a subsequence along which $p_t(F,G_n)$ converges to some $p_t(F)$
for every $t$ and every $F\in \Gr_t$.
Furthermore, as in Section~\ref{treecounts}, if $(G_n)$
has bounded tree counts then it is easy to check that
$\sum_{F\in \Gr_t}p_t(F)=1$, so $p_t$ is a probability distribution
on $\Gr_t$. Furthermore, these probability distributions for
different $t$ are consistent in the natural sense, and so may be combined
to form a probability distribution on $\Gr$. For this reason, we
say that a probability distribution $\pi$ on $\Gr$
is the {\em local limit} of the sequence $(G_n)$ if
\[
 p_t(F,G_n) \to \pi(\{G: G|_t=F\})
\]
for every $t\ge 1$ and every $F\in \Gr_t$, where for $G\in \Gr$, the graph
$G|_t$ is the subgraph of $G$ induced by the vertices within distance $t$ of the root.

Of course one can define a corresponding metric $\dloc$, by mapping each graph
to the point $(p_t(F,G_n))\in X=\prod_t [0,1]^{\Gr_t}$, and then applying
any metric on $X$ giving rise to the product topology.
Just as for tree counts, under suitable assumptions,
the local limit $\pi$ determines the limiting
subgraph counts $\ts$ and vice versa.
(One `suitable assumption' is bounded maximum degree. In fact,
as in Section~\ref{treecounts},
exponentially bounded tree counts will almost certainly do.)
Thus the metrics $\detl$ and $\dloc$ are equivalent.
As before, it seems more natural to study the distribution
of the neighbourhoods of vertices directly, rather than
the subgraph counts, which are essentially moments of this distribution.

This notion of local limit is extremely natural. In fact,
Benjamini and Schramm~\cite{BSrec} used the same notion (but for a random
rather than deterministic sequence ($G_n$)) to define a `distributional limit'
of certain random planar graphs; they showed that the random walk on the
limiting graph is recurrent with probability $1$.
The same notion in slightly different generality was studied
by Aldous and Steele~\cite{AS}, under the name `local weak limit',
and Aldous and Lyons~\cite{AL}, where the term `random weak limit' is used.
Returning to basic questions about this notion of limit, perhaps
the first is: which probability distributions on $\Gr$ can arise in this way?
As in Section~\ref{treecounts}, a necessary condition
is that the distribution $\pi$ must be {\em involution invariant}, meaning
that
\begin{equation}\label{io3}
 \E_\pi \sum_y f(G,x,y) = \E_\pi \sum_y f(G,y,x)
\end{equation}
for any non-negative isomorphism invariant function $f$ defined on triples $(G,x,y)$,
where $G$ is a locally finite graph and $x$ and $y$ are adjacent vertices of
$G$. Here the expectation is over the $\pi$-random rooted graph $(G,x)$,
and the sum is over neighbours of $x$.

The following question is due to Aldous and Lyons~\cite{AL}.

\begin{question}\label{qglim}
Does every involution-invariant probability distribution on $\Gr$
arise as the local limit of some sequence $(G_n)$ of graphs
with $|G_n|=n$?
\end{question}

Just as in the tree case, it may make sense to restrict to graphs
with bounded maximum degree, asking the analogue of Question~\ref{qtlimD}.
Note that it does not matter here whether we
consider a sequence of deterministic finite graphs, or a sequence
of distributions on $n$-vertex graphs: for the purposes
of Question~\ref{qglim}, a distribution on connected $n$-vertex graphs
may be well approximated by a much larger finite graph whose components
have approximately the right distribution.

Since this question seems to be rather important, let us briefly
describe its history; for more details we refer the reader to Aldous and Lyons~\cite{AL}.
Firstly, as noted above, the question is from~\cite{AL}, where it is
stated as an especially important open question.
(Lyons~\cite{Lyons} referred to a proof of a positive
answer to Question~\ref{qglim}, but in a note added in proof said
that this proof was incorrect.) 

Benjamini and Schramm~\cite{BSrec} were the first to note that any distribution
that is a local limit must be involution invariant. In
fact, they noted that it must satisfy an {\em a priori} stronger
condition they called the `intrinsic mass transport principle'. (This
is the same as involution invariance except that one considers a function
$f$ defined on triples $(G,x,y)$ where $x$ and $y$ are any vertices
of $G$, not necessarily adjacent vertices.)
Aldous and Steele~\cite{AS} introduced the somewhat simpler condition
of involution invariance. As shown by Aldous and Lyons~\cite{AL}, involution
invariance and the intrinsic mass transport principle are equivalent.

Aldous and Lyons~\cite{AL} use the term `unimodular' for
an involution invariant probability distribution on $\Gr$.
The motivation comes from a very special case: suppose
that $\pi$ is supported on a single rooted graph $(G,x)$. Then $G$
must be vertex transitive. In this case, $\pi$ is involution
invariant if and only if $G$ is unimodular,
in the sense that its automorphism group is unimodular, i.e.,
its left-invariant Haar measure is also right-invariant.

Unimodular transitive graphs have been studied for some time,
quite independently of the question of local limits (and
well before this arose);
%They are an extension of the
%class of amenable (transitive) graphs, and it turns out that many results
%that hold for amenable graphs but not general locally finite
%graphs do hold for unimodular transitive graphs;
see, for example,
Benjamini, Lyons, Peres and Schramm~\cite{BLPS_GAFA99}.
For a simple description of unimodularity in this context, see,
for example, Timar~\cite{Timar}.
It is perhaps surprising that there exist (bounded degree) vertex
transitive graphs that are non-unimodular. One example
is the `grandmother graph' $G$ shown in Figure~\ref{fig_gm},
introduced by Trofimov~\cite{Trofimov} in a slightly different context.
\begin{figure}[htb]
\centering
\epsfig{file=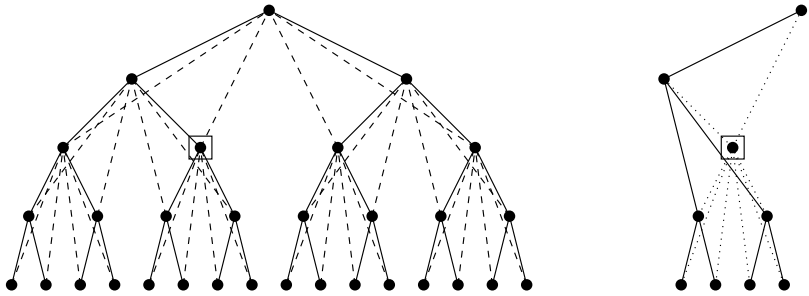,width=4.2in}
\caption{The grandmother graph (left), and the subgraph induced by the neighbourhood
of a vertex (right).}\label{fig_gm}
\end{figure}
This is the graph obtained by arranging the vertices
of the 3-regular tree into levels so that every vertex in level $n$
has one neighbour in level $n-1$, its parent, and two in level $n+1$, its children,
and then joining every vertex to its grandparent and grandchildren
in addition to its parent and children.
As can be seen from Figure~\ref{fig_gm}, from the 2-neighbourhood of a vertex
in $G$ one can identify its parent in the original 3-regular tree.
Let $\pi$ be the distribution on $\Gr$ supported only on the (rooted) grandmother graph $G$,
and let $f(G,x,y)$ be the function taking the value $1$ if $G$ is the grandmother graph
and $y$ is the parent of $x$, and $0$ otherwise; this function is isomorphism invariant.
For this $\pi$ and $f$, the left hand side of \eqref{io3} is the number of parents
of the root, i.e., 1, and the right hand side is the number of children (the number
of vertices whose parent the root is), i.e., 2. Thus $G$ is not unimodular.

Other examples of non-unimodular transitive graphs include the Diestel--Leader
graphs introduced in a different context in~\cite{DL}.

As noted by Aldous and Lyons~\cite{AL}, a positive answer to (their slightly
more general form of) Question~\ref{qglim} would have major implications
in group theory, since it would essentially imply that all finitely generated groups
are `sofic'. This group property was initially introduced (in a slightly different form)
by Gromov~\cite{Gromov};
the term `sofic' was coined by Weiss~\cite{Weiss}.
The key point is that several well-known conjectures in group theory
have been proved for sofic groups; see Elek and Szab\'o~\cite{ES04} for example.
For a brief survey of the topic of sofic groups, see Pestov~\cite{Pestov}.

Returning to metrics on sparse graphs,
in one sense, the metric $\dloc$ associated to the notion of local convergence
seems to be the most natural measure of `similarity' between general sparse
graphs. However, while this notion captures local information well,
it still loses global information: if $(G_n)$ has a certain
local limit $\pi$, then the graphs $H_{2n}$ formed by taking the disjoint
union of two copies of $G_n$ have the same local limit. (Indeed,
$p_t(F,G_n)=p_t(F,H_{2n})$ for every connected graph $F$.)
This shows that $\dloc$ fails to capture the global structure
of the graph, and suggests that it makes sense to consider
$\dloc$ and $\dP$ together; we shall do so in the next section.

\section{Further metrics, models and questions}\label{sec_further}

For fully dense graphs, with $\Theta(n^2)$ edges, the results of
Borgs, Chayes, Lov\'asz, S\'os and
Vesztergombi~\cite{BCLSV:1,BCLSV:2} show that one single metric, say
$\dc$, effectively captures several natural notions of local and
global similarity. Indeed, convergence in $\dc$ is equivalent to
convergence in the partition metric 
$\dP$ (a natural global notion) and to convergence in
$\de$, i.e., convergence of all small subgraph counts,
a natural local notion.

In the extremely sparse case, we have considered two metrics, the
partition metric $\dP$ and the local metric $\dloc$, which respectively
capture global and local similarity. Of course, one would like a single
metric capturing both notions, and also the interaction between local
and global properties. Fortunately, there is a natural combination
of $\dP$ and $\dloc$. 

\subsection{The coloured neighbourhood metric}\label{ss_cn}

Let $G_n$ be a graph with $n$ vertices, and $k\ge 1$ an integer.
We shall think of $G_n$ as having $\Theta(n)$ edges, though this
is only relevant when we come to sequences $(G_n)$.
Let $\Pi=(P_1,\ldots,P_k)$ be a partition of the vertex set of $G_n$,
which we may think of as a (not necessarily proper) $k$-colouring
of $G_n$. This time, for variety, we do not insist that the parts
have almost equal sizes; this makes essentially no difference.
Let $\Gr_{k,t}$ be the set of isomorphism classes of $k$-coloured
connected rooted graphs with radius at most $t$. For
each $F\in \Gr_{k,t}$, let
$p_{k,t}(G_n,\Pi)(F)$ be the probability that the $t$-neighbourhood
of a random vertex of the coloured graph $(G_n,\Pi)$
is isomorphic to $F$ as a coloured rooted graph,
so $p_{k,t}(G_n,\Pi)$ is a probability distribution on $\Gr_{k,t}$.
Finally, let
\[
 \M_{k,t}(G_n) =\{ p_{k,t}(G_n,\Pi) \},
\]
where $\Pi$ runs over all $k$-partitions of $V(G_n)$. Thus $\M_{k,t}(G_n)$
is a finite subset of the space $\P(\Gr_{k,t})$ of probability distributions
on $\Gr_{k,t}$. Of course, one can view $\M_{k,t}(G_n)$ as a multiset,
in which case it has exactly $k^n$ elements, one for each colouring.
However, as for the partition metric in Section~\ref{ss_par}, it turns
out to be better to ignore the multiplicities.

The space $\P(\Gr_{k,t})$ of probability distributions on $\Gr_{k,t}$
is naturally viewed as a metric space, with the total variation
distance $\dTV$ between two distributions as the metric. In other words,
regarding $\P(\Gr_{k,t})$ as a subset of the unit ball of $\ell_1$ in $\RR^{\Gr_{k,t}}$, we simply
take the $\ell_1$-metric on this set. Let $\dH$ denote the Hausdorff
distance between compact subsets of $\P(\Gr_{k,t})$, defined
with respect to $\dTV$. Then we may define the {\em coloured
neighbourhood metric} $\dcn$ by
\[
 \dcn(G,G') = \sum_{k\ge 1}\sum_{t\ge 1} 2^{-k-t} \dH(\M_{k,t}(G),\M_{k,t}(G')),
\]
say. (As before, we can instead use any metric
on $\prod_{t,k}\P(\Gr_{k,t})$ giving rise to the product topology.)
If we restrict our attention to graphs with maximum degree at most
some constant $\Delta$, then the corresponding sets $\Gr_{k,t}$
are finite, so each $\P(\Gr_{k,t})$ is compact, and any sequence $(G_n)$
has a subsequence that is Cauchy with respect to $\dcn$, and in fact
converges to a limit point consisting of one compact subset of $\P(\Gr_{k,t})$
for each $k$, $t$. In fact, it is not hard to check that whenever
$(G_n)$ has bounded tree counts (i.e., contains $O(n)$ copies
of any fixed tree $T$), it has a convergent
subsequence with respect to $\dcn$.
Of course, as before, we can combine the limiting subsets of $\P(\Gr_{k,t})$
as $t$ varies. Also, as in Section~\ref{ss_par}, there may be circumstances
in which it is better to view $\M_{k,t}(G_n)$ as a multiset after all,
and use the matching distance between such multisets to define
a {\em coloured neighbourhood matching metric} $\dcnM$.

Taking just $k=1$ above, we recover the notion of local
limit. On the other hand, the set $\M_k$ used to define the
partition metric can be recovered from $\M_{k,1}$. (The latter set
codes, for each partition, how many vertices there are in each part,
and how many neighbours in each part each vertex has. From this
information one can calculate the number of edges between each pair
of parts.) It follows that if $(G_n)$ is Cauchy with respect
to $\dcn$, then it is Cauchy with respect to both $\dloc$ and $\dP$.
In other words, $\dcn$ is a `joint refinement' of $\dloc$ and $\dP$.

\subsection{Models for metrics}

The following rather vague question was posed in~\cite{BRsparse}.

\begin{question}\label{qmodel}
Given a metric $d$, can we find a `natural' family of random graph models with the following
two properties:
(i) for each model, the sequence of random graphs $(G_n)$ generated
by the model is
Cauchy with respect to $d$ with probability $1$, and
(ii)
for any sequence $(G_n)$ with $|G_n|=n$ that is Cauchy with respect
to $d$, there is a model from the family
such that, if we interleave $(G_n)$ with a sequence of random
graphs from the model, the resulting sequence is still Cauchy with
probability $1$.
\end{question}

As noted in~\cite{BRsparse}, for any of $\dc$, $\de$ or $\dP$
the answer is yes in the dense case, since $(G_n)$ is Cauchy
if and only if $\dc(G_n,\ka)\to 0$
for some kernel $\ka$, while the dense inhomogeneous random graphs
$G(n,\ka)=G_1(n,\ka)$ converge to $\ka$ in $\dc$ with probability $1$.
Thus our family consists of one model $G(n,\ka)$ for each 
kernel $\ka$ (to be precise, for each equivalence
class of kernels under the relation $\sim$ defined in \ssequiv).
In the sparse case, but with $np\to\infty$, some partial
answers are given in~\cite{BRsparse}, but the situation
is much more complicated.

Here, with $p=1/n$, $G_{1/n}(n,\ka)$ is very unsatisfactory
as a model for an {\em arbitrary} sequence of sparse graphs,
since it produces graphs with essentially no cycles. The following natural
model proposed by Bollob\'as, Janson and Riordan~\cite{BJRclust}
is rather more general.
In the uniform case, generalizing $G(n,c/n)$,
assign a weight $w_F$ to each fixed graph $F$.  To generate a random graph
with $n$ vertices, starting from the empty graph, for each $F$ add
each of the $\Theta(n^{|F|})$ possible copies of $F$ with
probability $w_F / n^{|F|-1}$, deleting any duplicate edges.  Note
that, on average, we add $\Theta(n)$ copies of each graph $F$.
The point is that this model produces graphs with $\Theta(n)$ edges,
but (in general) $\Theta(n)$ triangles, and indeed $\Theta(n)$ copies
of any fixed graph $F$.

In the general case, Bollob\'as, Janson and Riordan~\cite{BJRclust}
start from a {\em kernel family} $(\ka_F)$
consisting of one kernel $\ka_F$ for each isomorphism type of connected finite
graph $F$; the kernel $\ka_F$ is simply a measurable function on $[0,1]^{V(F)}$
that is symmetric under the action of the automorphism group of $F$.
To construct the random graph $G(n,(\ka_F))$, choose 
$x_1,\ldots,x_n$ independently and uniformly from $[0,1]$,
and then for each $F$ and each set ${v_1,\ldots,v_k}$ of $k=|F|$ vertices,
insert a copy of $F$ with vertex set $v_1,\ldots,v_k$ with probability
$\ka_F(x_{v_1},\ldots,x_{v_k})/n^{k-1}$. For full details, see~\cite{BJRclust}. 

While the model $G(n,(\ka_F))$ is much more general than $G_{1/n}(n,\ka)$,
it still has its limitations. It is not hard to see that the asymptotic
degree distribution of $G(n,(\ka_F))$ will be a mixture of compound Poisson
distributions (rather than the mixture of Poisson distributions 
one obtains for $G_{1/n}(n,\ka)$).
In particular, there is no kernel family for which $G(n,(\ka_F))$ produces
graphs in which (almost) all vertices have degree $3$, say.
A positive answer to Question~\ref{qmodel} for either of the metrics $\dloc$
or $\dcn$ 
would involve, among other things, models that produce graphs with
arbitrary given degree distributions (with finite expectation, say, or
perhaps bounded, to keep things simple). Of course this is easy, but such an answer
would require much more -- it would require producing all possible
distributions of local structure. Even for $\dloc$, this is likely
to be hard, since one would presumably have to first understand
the possible limiting distributions, which
would involve answering Question~\ref{qglim}. For $\dcn$, much more
is needed: one needs to understand the possible combinations of local and global
structure.

Let us give one very simple example of the kind of model we have in mind,
associated to $\dloc$ in the extremely sparse case.
Let parameters $n$ and $d$ be given; we shall think of $d$ fixed
as $n\to \infty$. Assuming that $nd$ is a multiple of three, let $T(n,d)$
be the random graph obtained as follows: start with $n$ vertices, each
of which has $d$ `stubs' associated to it. Take a uniformly random partition
of the set of $nd$ stubs into $nd/3$ triplets, and add a triangle corresponding
to each triplet, sitting on the vertices that the stubs in the triplet are associated to.
In general, $T(n,d)$ is a multigraph,
but it will be very close to a simple graph (in fact, $T(n,d)$ will be simple
with probability bounded away from $0$). The model $T(n,d)$ is
a natural `triangle version' of the random regular graph, generated
via the configuration model. (Since the first draft of this paper
was written, Newman~\cite{Newman_tri} has described a natural
inhomogeneous version of this model.)

It is not hard to see that in $T(n,d)$ every (or, if we make the graph simple, almost
every) vertex has degree $2d$, and is in exactly $d$ (edge-disjoint) triangles.
Furthermore, other than this, $T(n,d)$ has no local structure: the local
limit of the sequence $T(n,d)$ is an infinite tree of triangles.
Thus $T(n,d)$ is an appropriate model for certain Cauchy sequences in $\dloc$.
Of course one can construct many other models along these lines, but it is
hard to imagine that all Cauchy sequences can be covered in this way!

As we have seen, in the extremely sparse case, Question~\ref{qmodel}
is likely to be very hard to answer for the metrics we have considered.
Nonetheless, it may be possible to answer the same question for weaker metrics,
or to provide partial answers. Such partial answers would hopefully provide
great insight into the structure of the set of sparse graphs.

\medskip
\begin{ack}
We are grateful to G\'abor Elek for pointing out an error in an earlier version
of this manuscript, and for drawing our attention to the connections
to the theory of sofic groups.
\end{ack}

\end{document}